\documentclass{amsart}
\usepackage[left=2.5cm, top=2.5cm,bottom=2.5cm,right=2.5cm]{geometry}
\usepackage{amsfonts}
\usepackage{latexsym}
\usepackage{amssymb}
\usepackage{amsmath}
\usepackage{color}
\usepackage{bbm}
\usepackage{tikz}
\usepackage{enumerate}
\usepackage{mathrsfs}

\newcommand{\R}{\mathbb R}
\newcommand{\N}{\mathbb N}

\newcommand{\E}{\mathbb E}
\newcommand{\Pro}{\mathbb P}

\newcommand{\Var}{\mathrm{Var}}
\newcommand{\Cov}{\mathrm{Cov}}

\def\dint{\textup{d}}
\newcommand{\SSS}{\ensuremath{{\mathbb S}}}

\newcommand{\B}{\ensuremath{{\mathbb B}}}

\newcommand{\rate}{\mathbb I}

\DeclareMathOperator{\vol}{vol}

\newtheorem{thm}{Theorem}[section]
\newtheorem{cor}[thm]{Corollary}
\newtheorem{lemma}[thm]{Lemma}

\newtheorem{rmk}[thm]{Remark}
\newtheorem{thmalpha}{Theorem}

\def\bC{\mathbf{C}}

\def\bP{\mathbf{P}}

\def\bU{\mathbf{U}}

\def\bW{\mathbf{W}}

\usepackage{nomencl}
\makenomenclature

\begin{document}


\title[LDP, MDP and the KLS conjecture]{Large deviations, moderate deviations,\\ and the KLS conjecture}

\author[D. Alonso-Guti\'errez]{David Alonso-Guti\'errez}
\address{Departamento de Matem\'aticas, Universidad de Zaragoza, Spain} \email{alonsod@unizar.es}

\author[J. Prochno]{Joscha Prochno}
\address{Institut f\"ur Mathematik \& Wissenschaftliches Rechnen, Karl-Franzens-Universit\"at Graz, Austria} \email{joscha.prochno@uni-graz.at}

\author[C. Th\"ale]{Christoph Th\"ale}
\address{Faculty of Mathematics, Ruhr University Bochum, Germany} \email{christoph.thaele@rub.de}

\keywords{Asymptotic geometric analysis, convex bodies, $\ell_p^n$-balls, KLS conjecture, large deviations principle, moderate deviations principle, random projections, stochastic geometry}
\subjclass[2010]{Primary: 60F10, 52A23 Secondary: 60D05, 46B09}

\begin{abstract}
Having its origin in theoretical computer science, the Kannan-Lov\'asz-Simonovits (KLS) conjecture is one of the major open problems in asymptotic convex geometry and high-dimensional probability theory today. In this work, we establish a new connection between this conjecture and the study of large and moderate deviations for isotropic log-concave random vectors, thereby providing a novel possibility to tackle the conjecture. We then study the moderate deviations for the Euclidean norm of random orthogonally projected random vectors in an $\ell_p^n$--ball. This leads to a number of interesting observations: (A) the $\ell_1^n$--ball is critical for the new approach; (B) for $p\geq 2$ the rate function in the moderate deviations principle undergoes a phase transition, depending on whether the scaling is below the square-root of the subspace dimensions or comparable; (C) for $1\leq p<2$ and comparable subspace dimensions, the rate function again displays a phase transition depending on its growth relative to $n^{p/2}$.
\end{abstract}

\maketitle


\section{Introduction and Results}

One of the major open problems in asymptotic convex geometry is the famous Kannan-Lov\'asz-Simonovits (KLS) conjecture. This conjecture has its origin in theoretical computer science and arose in the study of sampling algorithms for high-dimensional convex bodies. One specific situation of interest in the theory is to design an algorithm that efficiently computes the volume of an $n$-dimensional convex body. Such an algorithm is fed with a convex body $K\subseteq \R^n$ and a quality parameter $\varepsilon\in(0,\infty)$, where the body $K$ is represented by a membership oracle, which, for each $x\in \R^n$ can decide whether or not $x\in K$. The algorithm returns a number $V=V(K,\varepsilon)\in\R$ such that
\[
(1-\varepsilon)\vol_n(K) \leq V \leq (1+\varepsilon)\vol_n(K)
\]
and the efficiency of the designed algorithm is measured in terms of the number of arithmetic operations and calls to the membership oracle. While deterministic algorithms are deemed to be inefficient (see the works of B\'ar\'any and F\"uredi \cite{BF1987} and of Elekes \cite{Elekes1986}), there are reasonable \textit{randomized} algorithms available that can accurately compute the volume of convex bodies with high probability in polynomial time, see Dyer, Frieze, and Kannan \cite{DFK1991}. At its core, the construction of this randomized algorithm is connected to an isoperimetric inequality for log-concave probability measures on $\R^n$ (see, e.g., \cite{AGB2015} for an introduction to the KLS conjecture and \cite{VEMP2010}  for a more detailed explanation on its connection with problems in theoretical computer science). The constant appearing in this isoperimetric  inequality, known as Cheeger's constant, is of particular interest and directly linked to the KLS conjecture, which by the work of Rothaus, Cheeger, Maz'ya, and Ledoux (see, e.g., \cite[Theorem 1.1]{BH1997}) can be stated as follows:

\bigskip

\textbf{KLS Conjecture.} \textit{There exists an absolute constant $C\in(0,\infty)$ such that for all $n\in\N$, every centered random vector $X$ with log-concave distribution and any locally Lipschitz function $f:\R^n\to\R$ such that $f(X)$ is of finite variance,
\[
\Var[f(X)] \leq C\, \lambda_X^2\, \E\big[\| \nabla f(X)\|_2^2\big],
\]
where $\lambda_X^2:=\sup_{\theta\in\SSS^{n-1}}\E\langle X,\theta\rangle^2$.}

\bigskip

Using what is known as the localization lemma, Kannan, Lov\'asz, and Simonovits proved the conjecture with a factor $(\E\|X\|_2)^2$ instead of $\lambda_X^2$ (see \cite{KLS1995}), and later Bobkov \cite{BOB2007} improved this estimate. But besides the fact that the KLS conjecture holds for exponential and Gaussian random vectors and has attracted considerable attention in the past 25 years, there has been very little progress to the present day (see \cite{BOB2003} and \cite{HUET2011} for a proof of the conjecture for the class of revolution bodies, \cite{BW2009} for the case of the simplex, \cite{LW2008} and \cite{S2008} for the case of the $\ell_p^n$-balls, \cite{K2013} for the case of $1$-unconditional bodies with an extra $\log n$ factor, and \cite{LV2017} for the best general estimate of the constant of the order $\sqrt{n}$ in the KLS conjecture).  One of the reasons for the interest in the KLS conjecture is that, if correct, it would imply several other well-known conjectures in asymptotic convex geometry, including the variance conjecture (which is the special case $f(\cdot) = \|\cdot\|_2$) and the hyperplane conjecture. However, given the little progress in the KLS conjecture, it is natural to watch out for potential counterexamples. This brings us to one of the main purposes of this paper, which is to establish a new connection between the KLS conjecture and the study of so-called large and moderate deviations principles for isotropic log-concave random vectors that provides a potential route to disprove the variance conjecture and thus the KLS conjecture.

While the study of central limit theorems for random geometric quantities in asymptotic convex geometry is by now a classical and still flourishing part of the theory, as can be seen in \cite{APT2019, GKT2019, JP2019, KLZ2015, KPT2019_I, KPT2019_II, K2007_I, PPZ14, Schmu2001}, fluctuations beyond the Gaussian scale have only recently attracted attention in this context. For one, there are the large deviations, or more precisely large deviations principles, which determine the asymptotic likelihood of rare events on a scale of a law of large numbers, the most classical result in this direction being Cram\'er's theorem \cite{C1938} (see also \cite{dH2000}). Contrary to the universality shown by a central limit theorem, the large deviation behavior is typically sensitive to the distribution of the underlying random objects. Being a classical topic in probability theory, large deviations have only recently been introduced in asymptotic convex geometry and high-dimensional probability by Gantert, Kim, and Ramanan \cite{GKR2017}. They obtained a large deviations principle for $1$-dimensional projections of $\ell_p^n$-balls in $\R^n$ as the space dimension $n$ tends to infinity, showing how the speed and rate depend on the parameter $p$ and thus in a subtle way on the geometry of the body. Their work inspired a number of investigations regarding the large deviations behavior of quantities that naturally appear in asymptotic convex geometry, such as Sanov-type large deviations for the random spectral measure of a properly normalized matrix chosen uniformly at random in the unit ball of a Schatten $p$-class \cite{KPT2019_sanov} (a non-commutative analogue of a result previously obtained by Kim and Ramanan for random vectors in high-dimensional $\ell_p^n$-spheres \cite{KR2018}), large deviations principles for $\ell_q$-norms of random vectors in $\ell_p^n$-balls \cite{KPT2019_I, KPT2019_II}, a description of the large deviations behavior for the Euclidean norm of orthogonal projections of $\ell_p^n$-balls to high-dimensional random subspaces \cite{APT2018}, and, very recently, large deviations under an asymptotic thin-shell condition \cite{KR2019} as well as geometric sharp large deviations for random projections of $\ell_p^n$-spheres \cite{LR2020}. The question now arises of what can be said about the asymptotic likelihood of events on a scaling \emph{between} Gaussian fluctuations covered by the central limit theorem and large deviations? This scaling is exactly the one covered by moderate deviations principles, which, in contrast to large deviations, are typically non-parametric in rate. In fact, usually moderate deviations principles inherit properties from both the central limit theorem and the large deviations principle in that the central limit theorem is still visible in form of a Gaussian (quadratic) rate function, while probabilities decay on an exponential scale as in the large deviations principle. In the framework of random geometric quantities in asymptotic convex geometry those have first been considered by Kabluchko, Prochno, and Th\"ale in \cite{KPT2019_II} for $\ell_q$-norms of random vectors in $\ell_p^n$-balls.

Now, let us be more precise and introduce what a large or moderate deviations principle formally is and then describe its relation to the KLS conjecture. We recall that a sequence $(X_n)_{n\in\N}$ of random vectors in $d$-dimensional Euclidean space $\R^d$ satisfies a large deviations principle (LDP) with \emph{speed} $s_n$ and \emph{rate function} $\rate:\R^d\to[0,\infty]$ if
\begin{equation}\label{eq:ldp definition}
\begin{split}
-\inf_{x\in A^\circ}\rate(x) &\leq\liminf_{n\to\infty}s_n^{-1}\log\Pro[X_n\in A]\leq\limsup_{n\to\infty}s_n^{-1}\log\Pro[X_n\in A]\leq-\inf_{x\in\overline{A}}\rate(x)
\end{split}
\end{equation}
for all Borel measurable $A\subseteq\R^d$, where $A^\circ$ denotes the interior and $\overline A$ the closure of $A$, and where $\rate$ is lower semi-continuous and has compact level sets $\{x\in\R^d\,:\, \rate(x) \leq \alpha \}$, $\alpha\in\R$. So loosely speaking, an LDP says that if $n\in\N$ is sufficiently large and $A\subseteq\R^d$ sufficiently regular, then
\[
\Pro[X_n\in A] \approx e^{-s_n\, \inf\limits_{x\in A} \rate(x)}.
\]
A moderate deviations principle (MDP) formally resembles an LDP but on different scales and, typically, with important differences in the behavior of the two principles as described above. More precisely, the scales for an MDP are between that of a weak limit theorem (like a central limit theorem) and that of a law of large numbers.
In this paper, for two sequences $(x_n)_{n\in\N}$ and $(y_n)_{n\in\N}$ we use the Landau notation $x_n=o(y_n)$ if $\lim_{n\to\infty} \frac{x_n}{y_n}=0$ and $x_n=\omega(y_n)$ if $\lim_{n\to\infty}|\frac{x_n}{y_n}|=+\infty$. In addition, we write $x_n\approx y_n$ if there exist two absolute constants $c_1,c_2$ such that $c_1 x_n\leq y_n\leq c_2x_n$. The following result establishes a connection between the KLS conjecture and moderate/large deviations principles.

\begin{thmalpha}[LDPs/MDPs and the KLS conjecture]\label{thm:KLS}
Let $1\leq k_n\leq n$ be a sequence of integers such that $k_n=\omega(1)$ and let $(\xi_n)_{n=1}^\infty$ be a sequence of isotropic log-concave random vectors in $\R^{k_n}$. Consider a sequence of random variables $X_n=\frac{\Vert \xi_n \Vert_2}{\sqrt{k_n}}$, $n\in\N$, which satisfies \eqref{eq:ldp definition} with speed $s_n$ and rate function $\rate$. Assume that one of the following two conditions is satisfied:
\begin{itemize}
\item[(a)] $s_n=o(\sqrt{k_n})$ and $\rate$ is non-singular, i.e., $\rate(x)\neq\rate_0(x)=
\begin{cases}
0&: x=1\\
+\infty&:\text{otherwise}.
\end{cases}$
\item[(b)] $s_n\approx \sqrt{k_n}$ and $\inf\limits_{t>t_0}\frac{\inf_{x\in(t,\infty)}\rate(x)}{t}=0$ for some absolute constant $t_0\in(1,\infty)$.
\end{itemize}
Then the KLS conjecture is false.
\end{thmalpha}

For a sequence of random vectors $\xi_n\in\R^{k_n}$, $n\in\N$, as given in Theorem \ref{thm:KLS} it is well known that $\E\|\xi_n\|_2/\sqrt{k_n}\to 1$, as $n\to\infty$. Against this light, the natural scale for a law of large numbers and hence an LDP is $\sqrt{k_n}$, while scales with $t_n=o(\sqrt{k_n})$ and $t_n=\omega(1)$ correspond to an MDP. Thus, Theorem \ref{thm:KLS} establishes a connection between LDPs or MDPs for isotropic log-concave random vectors and the KLS conjecture, showing that by constructing a sequence of random vectors $\xi_n$, $n\in\N$, such that $X_n$, $n\in\N$, follows an LDP or MDP with an appropriate speed and/or rate function would disprove the variance and hence the KLS conjecture.

In this context it is instructive to recall that the large deviations principle \cite[Theorem 1.2]{APT2018} or \cite[Theorem D]{KPT2019_I} for the Euclidean norm of uniform random points in the rescaled $n$-dimensional crosspolytope $\mathbb{D}_1^n:=\sqrt{(n+1)(n+2)\over 2}\B_1^n$ (the normalization is designed in such a way that a uniform random vector in $\mathbb{D}_1^n$ is isotropic) holds with $s_n=\sqrt{n}$ and rate function
$$
\rate(x) = \begin{cases}
\sqrt{x^2-1} &: x>1\\
+\infty &: \text{otherwise}.
\end{cases}
$$
In fact, the result from \cite{APT2018} can be transferred to our rescaled situation by a contraction with the function $x\mapsto x/\sqrt{2}$.
This implies that, for any $t_0>1$, the function
$$
\frac{\sqrt{t^2-1}}{t}=\sqrt{1-\frac{1}{t^2}}
$$
is equivalent to a constant in the interval $[t_0,\infty)$, since
$$
\inf_{t>t_0}{\sqrt{1-\frac{1}{t^2}}} = \sqrt{1-\frac{1}{t_0^2}}\quad\text{and}\quad\sup_{t>t_0}\sqrt{1-\frac{1}{t^2}}=1.
$$
Thus, we cannot find $t_0>1$ such that the condition in Theorem \ref{thm:KLS} (b) is satisfied while at the same time there exists no $t_0>1$ and an exponent $\alpha>1$ such that $\inf_{t>t_0}\frac{\inf_{x\in(t,\infty)}\rate(x)}{t^\alpha}>C$ for some absolute constant $C\in(0,\infty)$. This shows that (rescaled) crosspolytopes constitute a critical case in Theorem \ref{thm:KLS} as they \textit{just} fail to satisfy condition (b) there. Notice also that the case of the crosspolytope shows that the condition $t_0\in(1,\infty)$ in (b) cannot be relaxed to $t_0\in[1,\infty)$ as this would give $\inf_{t>1}\frac{\inf_{x\in(t,\infty)}\rate(x)}{t}=0$ while we know that they do satisfy the KLS conjecture (see \cite[Theorem 2.3]{AGB2015}).

\medspace

The crosspolytopes just discussed are members of the important class of $\ell_p^n$-balls, whose probabilistic aspects have been under intensive investigation during the last years \cite{APT2018,APT2019,GKR2017,GRR2014, JP2019,KPT2019_I,KPT2019_II,KR2018,LR2020,KR2019} (see also the comments in the introduction). As usual, for $1\leq p<\infty$ we let $\B_p^n :=\{x\in\R^n: \|x\|_p\leq 1\}$ denote the unit ball in $\R^n$ with respect to the $\ell_p^n$-norm $\|x\|_p := (|x_1|^p+\ldots +|x_n|^p)^{1/p}$, $x=(x_1,\ldots,x_n)\in\R^n$. In the remaining parts of this paper we are studying random vectors that arise as orthogonal projections of uniformly distributed random points in $\B_p^n$ onto random $k_n$-dimensional subspaces. To formally introduce the framework and the necessary notation, let $(k_n)_{n\in\N}$ be a sequence of integers with $1\leq k_n\leq n$. Further, for each $n\in\N$ let $X_n$ be a uniformly distributed random vector in $\B_p^n$. Consider a random $k_n$-dimensional subspace $E_n$ of $\R^n$, which is distributed according to the Haar probability measure on the Grassmannian manifold of all $k_n$-dimensional subspaces of $\R^n$ and assumed to be independent of $X_n$. By $P_{E_n}$ we denote the orthogonal projection onto $E_n$ and introduce the random variables
$$
\mathscr{Z}_{n,p} := n^{1/p}\|P_{E_n}X_n\|_2,\qquad n\in\N,
$$
 where the factor $n^{1/p}$ is motivated by the asymptotic order of the factor $\vol(\B_p^n)^{1/n}$ by which one has to divide $X_n$ in order to make it, and consequently $P_{E_n}X_n$, isotropic. Our next result is an MDP for $\mathscr{Z}_{n,p}$ on the critical scale $\sqrt{k_n}$, i.e., for $k_n^{-{1/2}}\mathscr{Z}_{n,p}$, whose study is motivated by Theorem \ref{thm:KLS}. We will see that in this particular situation the rate function is neither universal nor given by the quadratic Gaussian rate function. Instead, the rate function for $1\leq p<2$ displays an interesting and non-expected phase transition, which is controlled by the relative growth of the subspace dimensions $k_n$ to $n^{p/2}$.

\begin{thmalpha}[MDP on the critical scale]\label{thm:CriticalScale}
Let $1\leq p<\infty$ and $(k_n)_{n\in\N}$ be a sequence of positive integers $1\leq k_n\leq n$ such that $k_n=\omega(1)$ and $k_n=o(n)$. Let $M_p(2)=\frac{p^{2/p}}{3}\frac{\Gamma\left(1+\frac{3}{p}\right)}{\Gamma\left(1+\frac{1}{p}\right)}$. Then the following hold:
\begin{itemize}
\item[(a)] If $p\geq 2$, then $k_n^{-{1/2}}\mathscr{Z}_{n,p}$ satisfies an MDP on $\R$ with speed $k_n$ and rate function
$$
\rate(x) = \begin{cases}
{x^2-M_p(2)\over 2M_p(2)}-\log \frac{x}{\sqrt{M_p(2)}} &: x>0\\ +\infty &: \text{otherwise}.
\end{cases}
$$
\item[(b1)] If $1\leq p<2$ and $k_n=o(n^{p/2})$, then $k_n^{-{1/2}}\mathscr{Z}_{n,p}$ satisfies an MDP on $\R$ with speed $k_n$ and rate function
$$
\rate(x) =\begin{cases}
\frac{x^2-M_p(2)}{2}-\log \frac{x}{\sqrt{M_p(2)}} &: x>0\\
+\infty&:\text{otherwise}.
\end{cases}
$$
\item[(b2)] If $1\leq p<2$ and $k_n=n^{p/2}$, then $k_n^{-{1/2}}\mathscr{Z}_{n,p}$ satisfies an MDP on $\R$ with speed $n^{p/2}$ and rate function
$$
\rate(x) =\begin{cases}
\inf\limits_{y\geq m_p\atop}\frac{\left(\frac{x}{y}\right)^2-1}{2}-\log \left(\frac{x}{y}\right)+\frac{1}{p}(y^2-M_p(2))^\frac{p}{2}&: x>0\\
+\infty &:\text{otherwise}.
\end{cases}
$$
\item[(b3)] If $1\leq p<2$ and $k_n=\omega(n^{p/2})$, then $k_n^{-{1/2}}\mathscr{Z}_{n,p}$ satisfies an MDP on $\R$ with speed $n^{p/2}$ and rate function
$$
\rate(x) = \begin{cases}
\frac{1}{p}(x^2-M_p(2))^\frac{p}{2} &: x>\sqrt{m_p}\\
+\infty&:\text{otherwise}.
\end{cases}
$$
\end{itemize}
\end{thmalpha}

After having discussed the moderate deviations at the critical scale $\sqrt{k_n}$, we finally turn to a description of the moderate deviations at subcritial scales $o(\sqrt{k_n})$, which complement the central limit theorem and the large deviations principle obtained in \cite{APT2019} and \cite{APT2018}, respectively. In this case, we will see that the rate function is indeed the universal quadratic Gaussian rate function also visible in the central limit theorem, but with a prefactor encoding geometric information. In order to simplify the presentation, let us introduce for $n\in\N$ the random variable
$$
\mathscr{X}_{n,p}:=n^{1/p}\sqrt{{\Gamma\big({1\over p}\big)\over p^{2/p}\Gamma\big({3\over p}\big)}}\,\|P_{E_n}X_n\|_2-\sqrt{k_n},
$$
which is normalized in such a way that $\E[\mathscr{X}_{n,p}]\to0$, as $n\to\infty$. For technical reasons, we also restrict the following result to the case that $p\geq 2$.

\begin{thmalpha}[MDP on subcritical scales]\label{thm:MDP}
Let $2\leq p<\infty$ and $(k_n)_{n\in\N}$ be a sequence of integers such that $1\leq k_n\leq n$ and $\lim\limits_{n\to\infty}{k_n\over n}=\lambda\in[0,1]$. Assume that $k_n=\omega(1)$ and that the sequence $(t_n)_{n\in\N}$ of positive real numbers satisfies $t_n=\omega(1)$, $t_n=o(\sqrt{k_n})$ {and either $t_n=o(\sqrt{n-k_n})$ or $(n-k_n)=o\big(\frac{\sqrt{k_n}}{t_n}\big)$.} Then the sequence of random variables $t_n^{-1}\mathscr X_{n,p}$ satisfies an MDP on $\R$ with speed $t_n^2$ and rate function $\rate(x)=\alpha_{p,\lambda}x^2$ with constant $\alpha_{p,\lambda}$ given by
\begin{equation}\label{eq:ConstantAlpha}
\alpha_{p,\lambda} = {2\Gamma({1\over p})\Gamma({3\over p})^2\over (2p-\lambda(4+3p))\Gamma(1+{1\over p})\Gamma({3\over p})^2+\lambda\Gamma({1\over p})^2\Gamma({5\over p})} .
\end{equation}
\end{thmalpha}

\begin{rmk}\label{rem:intro}\rm
We end this section with a number of comments.

\begin{itemize}
\item[(a)] Although the results we presented so far were formulated for uniform random vectors $X_n$ in $\B_p^n$, we will prove more general versions for random vectors that are distributed in $\B_p^n$ according to a distribution from the class introduced in \cite{BGMN2005}. As special cases this includes the uniform probability measure or the cone probability measure on $\B_p^n$, but also image measures of these measures on higher-dimensional $\ell_p^n$-balls under coordinate projections. We refer to the discussion in the next section.

\item[(b)] We would like to point out that the constant $\alpha_{p,\lambda}$ in \eqref{eq:ConstantAlpha} coincides with $1\over 2\sigma^2(p,\lambda)$, where $\sigma^2(p,\lambda)$ is the variance of the centered Gaussian random variable appearing as the limit in the central limit theorem for $\mathscr{X}_{n,p}$ from \cite{APT2019}. In other words, this means that the rate function $\rate(y)$ in Theorem \ref{thm:MDP} can be rewritten as $\rate(y)={y^2\over 2\sigma^2(p,\lambda)}$, as expected.

\item[(c)] {We leave it as an open problem for future research to study moderate deviations at the supercritical scales $t_n=\omega(\sqrt{k_n})$ (and $t_n=o(\sqrt{n})$). In addition, it would be interesting to investigate the situation where the involved parameters neither satisfy $t_n=o(\sqrt{n-k_n})$ nor $(n-k_n)=o\big(\frac{\sqrt{k_n}}{t_n}\big)$. Another open problem is to decide whether or not the rate function in the MDP in Theorem \ref{thm:MDP} is the same also for $1\leq p<2$. }
\end{itemize}
\end{rmk}

\section{Preliminaries}

\subsection{Notation}

We supply the $n$-dimensional Euclidean space $\R^n$, where $n\in\N$, with the Euclidean norm $\|\,\cdot\,\|_2$ and the standard scalar product $\langle\,\cdot\,,\,\cdot\,\rangle$.  The $\ell_p^n$-norm will be denoted by $\|\,\cdot\,\|_p$ and $\B_p^n:=\{x\in\R^n:\|x\|_p\leq 1\}$ will denote the $n$-dimensional unit ball with respect to the $\ell_p^n$-norm. For a set $A\subset\R^d$ we shall write $A^\circ$ for the interior and $\overline{A}$ for the closure of $A$.

\subsection{Distributions on $\ell_p^n$-balls}

As already mentioned in Remark \ref{rem:intro}, we consider a much more general class of distributions compared to \cite{APT2018}, \cite{GKR2017}, \cite{KPT2019_I} and \cite{KR2018}. Those have been intensively studied by Barthe, Gu\'edon, Mendelson, and Naor \cite{BGMN2005}, and are closely related to the geometry of $\ell_p^n$-balls. This class contains the uniform distribution as well as the cone probability measure on the $\ell_p^n$-unit ball $\B_p^n :=\{x\in\R^n: \|x\|_p\leq 1\}$ considered in \cite{APT2019,GKR2017,KPT2019_I} as special cases, and many more (see below). As usual, $\|x\|_p := (|x_1|^p+\ldots +|x_n|^p)^{1/p}$ denotes the $\ell_p^n$-norm of the vector $x=(x_1,\ldots,x_n)\in\R^n$, and the parameter $p$ satisfies $1\leq p<\infty$.  For fixed space dimension $n\in\N$, we let $\bW$ be any Borel probability measure on $[0,\infty)$, $\bU_{n,p}$ be the uniform distribution on $\B_p^n$, and $\bC_{n,p}$ be the cone probability measure on the boundary of $\B_p^n$, where we recall that
$\bC_{n,p}(A) := \bU_{n,p}(\{rx:x\in A,0\leq r\leq 1\})$ for a measurable subset $A$ on the boundary $\SSS_p^{n-1}=\{x\in\R^n:\|x\|_p=1\}$ of $\B_p^n$.
 The distributions we consider are of the form
\begin{equation}\label{eq:DefMeasurePnpW}
\bP_{\bW,n,p} := \bW(\{0\})\,\bC_{n,p} + H\,\bU_{n,p},
\end{equation}
where the function $H:\B_p^n\to\R$ is given by $H(x)=h(\|x\|_p)$ with
$$
h(r) = {1\over p^{n/p}\Gamma\big(1+{n\over p}\big)}{1\over (1-r^p)^{1+n/p}}\int_0^\infty s^{n/p}e^{-{1\over p}{s r^p (1-r^p)^{-1}}}\,\bW(\dint s),\qquad r\in[0,1].
$$
In other words this means that
\begin{align*}
\int_{\B_p^n}f(x)\,\bP_{\bW,n,p}(\dint x) &= \bW(\{0\})\int_{\SSS_p^{n-1}}f(x)\,\bC_{n,p}(\dint x) + \int_{\B_p^n}f(x)\,H(x)\,\bU_{n,p}(\dint x)\\
&=\bW(\{0\})\int_{\SSS_p^{n-1}}f(x)\,\bC_{n,p}(\dint x) + \int_{\B_p^n}f(x)\,h(\|x\|_p)\,\bU_{n,p}(\dint x)
\end{align*}
for all non-negative measurable functions $f:\B_p^n\to\R$. The class of measures of the form $\bP_{\bW,n,p}$ contains the following important cases, which are of particular interest (see Theorem 1, Theorem 3, Corollary 3, and Corollary 4 in \cite{BGMN2005}):
\begin{itemize}
\item[(i)] If $\bW$ is the exponential distribution with rate $1/p$ (and mean $p$), then $\bW(\{0\})=0$, $H\equiv 1$, and $\bP_{\bW,n,p}$ reduces to the uniform distribution $\bU_{n,p}$ on $\B_p^n$.
\item[(ii)] If $\bW=\delta_0$ is the Dirac measure concentrated at $0$, then $\bW(\{0\})=1$, $H\equiv 0$, and $\bP_{\bW,n,p}$ is just the cone probability measure on $\B_p^n$.
\item[(iii)] If $\bW={\rm Gamma}(\alpha,1/p)$ is a gamma distribution with shape parameter $\alpha>0$ and rate $1/p$, then $\bP_{\bW,n,p}$ is the beta-type probability measure on $\B_p^n$ with Lebesgue density given by
$$
x\mapsto
{\Gamma\big(\alpha+{n\over p}\big)\over\Gamma(\alpha) \big(2 \Gamma\big(1+{1\over p}\big)\big)^n}\,\big(1-\|x\|_p^p\big)^{\alpha-1},
\qquad x\in\B_p^n\,.
$$
In particular, if $\alpha=m/p$ for some $m\in\N$, this is the image of the cone probability measure $\bC_{n+m,p}$ on $\B_p^{n+m}$ under the orthogonal projection onto the first $n$ coordinates. Similarly, if $\alpha=1+m/p$, this distribution arises as the image of the uniform distribution $\bU_{n+m,p}$ on $\B_p^{n+m}$ under the same orthogonal projection.
\end{itemize}

\subsection{Gaussian and $p$-generalized Gaussian random variables}

Let us denote, for $1\leq p<\infty$, by $(Z_i)_{i\in\N}$ a sequence of independent copies of a $p$-generalized Gaussian random variable with Lebesgue density
$$
f_p(x)=\frac{e^{-\frac{|x|^p}{p}}}{2p^{1/p}\Gamma\left(1+1/p\right)},\qquad x\in\R.
$$
Defining
\begin{equation}\label{eq:Mpq}
M_p(q):=\frac{p^{q/p}}{q+1}\frac{\Gamma\left(1+\frac{q+1}{p}\right)}{\Gamma\left(1+1/p\right)}
\end{equation}
for $q\geq 1$, we can express the moments of $p$-generalized Gaussian random variables as follows:
\begin{align}\label{eq:Moments}
\E|Z_1|^q = M_p(q)\qquad\text{and}\qquad\Cov(|Z_1|^r,|Z_1|^s)=M_p(r+s)-M_p(r)M_p(s),
\end{align}
where $q,r,s\geq 1$, see \cite[Lemma 3.1]{APT2019}. In particular, $\Var|Z_1|^q=M_p(2q)-M_p(q)^2$. We shall use these relations frequently in the proof of Theorem \ref{thm:MDP} and its general version Theorem \ref{thm:MDPgen}.

If $p=2$, then a $p$-generalized Gaussian random variable is nothing else than a standard Gaussian random variable. In this text we shall denote a sequence of such independent random variables by $(g_i)_{i\in\N}$.

\subsection{Background material from large deviations theory}

Let $(\Omega,\mathcal A,\Pro)$ be a probability space and $(X_n)_{n\in\N}$ be a sequence of random vectors taking values in $\R^d$ for some $d\geq 1$. Further, let $(s_n)_{n\in\N}$ be an increasing sequence of real numbers and $I:\R^d\to[0,\infty]$ be a lower semi-continuous function with compact level sets $\{x\in\R^d\,:\, I(x) \leq \alpha \}$, $\alpha\in\R$. One says that $(X_n)_{n\in\N}$ satisfies a large deviations principle (LDP) on $\R^d$ with speed $s_n$ and rate function $I$ provided that
\begin{equation*}
\begin{split}
-\inf_{x\in A^\circ}I(x) &\leq\liminf_{n\to\infty}s_n^{-1}\log\Pro[X_n\in A]\leq\limsup_{n\to\infty}s_n^{-1}\log\Pro[X_n\in A]\leq-\inf_{x\in\overline{A}}I(x)
\end{split}
\end{equation*}
for all Borel measurable $A\subseteq\R^d$. Especially, if $A$ is an $I$-continuity set, that is, if $A$ satisfies $I(A^\circ)=I(\bar{A})$ with $I(A):=\inf\{I(x):x\in A\}$, one has the exact limit relation
$$
\lim_{n\to\infty}s_n^{-1}\log\Pro[X_n\in A]=-I(A).
$$

Let $d_1,d_2\in\N$ and suppose that $(X_n)_{n\in\N}$ is a sequence of $\R^{d_1}$-valued random vectors and that $(Y_n)_{n\in\N}$ is a sequence of $\R^{d_2}$-random vectors. We assume that both sequences satisfy LDPs with the same speed. The next result, taken from \cite[Proposition 2.4]{APT2018}, yields that also the sequence of $\R^{d_1+d_2}$-valued random vectors $(X_n,Y_n)$ satisfies an LDP and provides the form of the rate function.

\begin{lemma}[LDP for independent vectors]\label{JointRateFunction}
Assume that $(X_n)_{n\in\N}$ satisfies an LDP on $\R^{d_1}$ with speed $s_n$ and rate function $I_X$ and that $(Y_n)_{n\in\N}$ satisfies an LDP on $\R^{d_2}$ with speed $s_n$ and rate function $I_Y$. Then, if $X_n$ and $Y_n$ are independent for each $n\in\N$, the sequence of random vectors $(X_n,Y_n)$ satisfies an LDP on $\R^{d_1+d_2}$ with speed $s_n$ and rate function $I$ given by $I(x):=I_X(x_1)+I_Y(x_2)$, $x=(x_1,x_2)\in\R^{d_1}\times\R^{d_2}$.
\end{lemma}

Next, assume that a sequence $(X_n)_{n\in\N}$ of random variables satisfies an LDP with speed $s_n$ and rate function $I$. Suppose now that $(Y_n)_{n\in\N}$ is a sequence of random variables that are `close' to the ones from the first sequence. The next result provides conditions under which in such a situation an LDP from the first can be transferred to the second sequence, see \cite[Theorem 4.2.13]{DZ2010} or \cite[Lemma 27.13]{Kallenberg}.

\begin{lemma}[Exponential equivalence]\label{prop:exponentially equivalent}
Let $(X_n)_{n\in\N}$ and $(Y_n)_{n\in\N}$ be two sequence of $\R^d$-valued random vectors and assume that $(X_n)_{n\in\N}$ satisfies an LDP on $\R^d$ with speed $s_n$ and rate function $I$. Further, suppose that the two sequences $(X_n)_{n\in\N}$ and $(Y_n)_{n\in\N}$ are exponentially equivalent, i.e.,
$$
\limsup_{n\to\infty}s_n^{-1}\log\Pro\big[\|X_n-Y_n\|_2>\delta\big] = -\infty
$$
for any $\delta>0$. Then $(Y_n)_{n\in\N}$ satisfies an LDP on $\R^d$ with the same speed and the same rate function.
\end{lemma}

Let us now recall what is known as Cram\'er's theorem for sequences of real-valued random variables. It provides an LDP for sequences of independent and identically distributed random variables, see \cite[Theorem 2.2.3]{DZ2010}. The rate function in Cram\'er's theorem is identified as the Legendre-Fenchel transform of the cumulant generating function of the involved random variables.

\begin{lemma}[Cram\'er's theorem]\label{lem:cramer}
Let $(X_n)_{n\in\N}$ be a sequence of independent and identically distributed random variables. Assume that $\E[e^{\lambda X_1}]<\infty$ for some $\lambda>0$. Then the sequence of random variables ${1\over n}\sum_{i=1}^n X_i$ satisfies an LDP on $\R$ with speed $n$ and rate function $I(x)=\sup\big\{\lambda x-\log\E [e^{\lambda X_1}]:\lambda\in\R\big\}$.
\end{lemma}

Finally, we consider the possibility to transport a large deviations principle to another one by means of a a sequence of functions. We recall the following version of the contraction principle from \cite[Corollary 4.2.21]{DZ2010}.

\begin{lemma}[Contraction principle]\label{lem:refinement contraction principle}
Let $d_1,d_2\in\N$ and let $F:\R^{d_1}\to\R^{d_2}$ be a continuous function. Suppose that $(X_n)_{n\in\N}$ is a sequence of $\R^{d_1}$-valued random variables that satisfies an LDP on $\R^{d_1}$ with speed $s_n$ and rate function $I$. Further, suppose that for each $n\in\N$, $F_n:\R^{d_1}\to\R^{d_2}$ is a measurable function such that for all $\delta>0$,
$$
\limsup_{n\to\infty}s_n^{-1}\log\Pro[X_n\in\Gamma_{n,\delta}]=-\infty,
$$
where $\Gamma_{n,\delta}:=\{x\in\R^{d_1}:\|F_n(x)-F(x)\|_2>\delta\}$. Then the sequence of $\R^{d_2}$-valued random variables $(F_n(X_n))_{n\in\N}$ satisfies an LDP on $\R^{d_2}$ with the same speed and with rate function $I\circ F^{-1}$.
\end{lemma}

\subsection{Moderate deviations in $\R^d$}

A moderate deviations principle (MDP) is formally nothing else than a large deviations principle. As already explained in the introduction, the difference is that LDPs provide estimates on the scale of a law of large numbers, while MDPs typically describe the probabilities at scales between a law of large numbers and a central limit theorem.
An important tool for us will be the following MDP for sums of independent and identically distributed random vectors, see \cite[Theorem 3.7.1]{DZ2010}.

\begin{lemma}[MDP for sums of i.i.d.\ random vectors]\label{lem:MDPVector}
Let $(X_n)_{n\in\N}$ be a sequence of independent and identically distributed random vectors in $\R^d$ and let $(t_n)_{n\in\N}$ be sequence of positive real numbers such that {$t_n=\omega(1)$} and $t_n=o(\sqrt{n})$. We assume that $X_1$ is centered, its covariance matrix $\bC=\Cov(X_1)$ is invertible, and $\log\E[e^{\langle \lambda,X_1\rangle}]<\infty$ for all $\lambda$ in a ball around the origin having positive radius. Then the sequence of random vectors $\frac{1}{t_n\sqrt{n}}\sum_{i=1}^nX_i$, $n\in\N$, satisfies an LDP (i.e.,\ an MDP as the sum is scaled by $t_n\sqrt{n}$) with speed $t_n^2$  and rate function $I(x)={1\over 2}\langle x,\bC^{-1}x\rangle$, $x\in\R^d$. 
\end{lemma}

\section{Proof of Theorem \ref{thm:KLS}}

Let us recall from \cite[Theorem 1.15]{AGB2015} (see also \cite{GM1987} and \cite{L1994}) that if the KLS conjecture were true, there would exist an absolute constant $C\in(0,\infty)$ such that, for every $n\in\N$ and all $t>0$,
$$
\Pro\left[\left|\frac{\Vert \xi_n\Vert_2}{\sqrt{k_n}}-1\right|>t\right]\leq 2e^{-Ct\sqrt{k_n}}.
$$
Therefore, we have 
\begin{eqnarray*}
\log \Pro\left[\frac{\Vert \xi_n\Vert_2}{\sqrt{k_n}}-1>t\right]&\leq&\log\Pro\left[\left|\frac{\Vert \xi_n\Vert_2}{\sqrt{k_n}}-1\right|>t\right]\leq \log 2-Ct\sqrt{k_n}
\end{eqnarray*}
and
\begin{eqnarray*}
\log \Pro\left[\frac{\Vert \xi_n\Vert_2}{\sqrt{k_n}}<1-t\right]&\leq&\log\Pro\left[\left|\frac{\Vert \xi_n\Vert_2}{\sqrt{k_n}}-1\right|>t\right]\leq \log 2-Ct\sqrt{k_n}
\end{eqnarray*}
and thus, for every $n\in\N$ and $t>0$,
\begin{eqnarray*}
\frac{\log \Pro\left[\frac{\Vert \xi_n\Vert_2}{\sqrt{k_n}}>1+t\right]}{s_n}&\leq&\frac{\log 2}{s_n}-\frac{Ct\sqrt{k_n}}{s_n}
\end{eqnarray*}
and
\begin{eqnarray*}
\frac{\log \Pro\left[\frac{\Vert \xi_n\Vert_2}{\sqrt{k_n}}<1-t\right]}{s_n}&\leq&\frac{\log 2}{s_n}-\frac{Ct\sqrt{k_n}}{s_n}.
\end{eqnarray*}
Therefore, taking the limit inferior, as $n\to\infty$, and taking into account the assumption that the sequence of random variables $\frac{\Vert X_n \Vert_2}{\sqrt{k_n}}$ satisfies \eqref{eq:ldp definition} {with speed $s_n$ and rate function $\rate$}, we make the following observation: if $s_n=o(\sqrt{k_n})$, then for every $t>0$,
$$
-\inf_{x\in(1+t,\infty)}\rate(x)\leq -\infty \quad \text{and}\quad -\inf_{x\in(-\infty,1-t)}\rate(x)\leq -\infty,
$$
which implies that $\rate$ is identically equal to $+\infty$ on $\R\setminus\{1\}$. Since
$$
\Pro\left[\frac{\Vert \xi_n\Vert_2}{\sqrt{k_n}}\in\R\right]=1,
$$
we have that, for every $n\in\N$,
$$
\frac{\log \Pro\left[\frac{\Vert \xi_n\Vert_2}{\sqrt{k_n}}\in\R\right]}{s_n}=0.
$$
Hence, taking the limit as $n\to\infty$,
$$
0 = -\inf_{x\in\R}\rate(x)=-\rate(1),
$$
{which implies that $\rate$ would coincide with the singular rate function $\rate_0$, a contradiction to our assumption that $\rate\neq\rate_0$.}

If otherwise
$s_n\approx\sqrt{k_n}$, then for any $t>0$ there exists a constant $C_1\in(0,\infty)$ such that
$$
-\inf_{x\in(t+1,\infty)}\rate(x)\leq -C_1t,
$$
which implies that
$$
\frac{\inf_{x\in(t+1,\infty)}\rate(x)}{t+1}>C_1{t\over t+1}
$$
for all $t>0$. This is equivalent to the fact that, for every $t>1$,
$$
\frac{\inf_{x\in(t,\infty)}\rate(x)}{t}>C_1{t-1\over t}.
$$
Now, for every $t_0>1$ we have that for all $t\in[t_0,\infty)$, $C_1{t-1\over t}>C_2(t_0)$. Thus, for any $t_0>1$,
$$
\inf_{t>t_0}\frac{\inf_{x\in (t,\infty)} \rate(x)}{t} \geq {C_2(t_0)}.
$$
However, this is a contradiction to our assumption.\hfill $\Box$

\section{Proof of Theorem \ref{thm:CriticalScale} and its generalization}

In this section we prove the following generalized version of Theorem \ref{thm:CriticalScale}. For this recall the definition of the probability measures $\bP_{\bW,n,p}$ on $\B_p^n$.

\begin{thm}[MDP on the critical scale, general version]\label{thm:CriticalScalegen}
Let $1\leq p<\infty$ and $(k_n)_{n\in\N}$ be a sequence of positive integers such that $1\leq k_n\leq n$, $k_n=\omega(1)$, and $k_n=o(n)$. Also let $(\bW_n)_{n\in\N}$ be a sequence of probability measures on $[0,\infty)$ and for each $n\in\N$ let $X_n$ be a random vector with distribution $\bP_{\bW_n,n,p}$. Independently of $X_n$, let $E_{n}$ be a uniformly distributed $k_n$-dimensional random subspace for each $n\in\N$, and define $\mathscr{Z}_{n,p}:=n^{1/p}\|P_{E_n}X_n\|_2$. Suppose that the sequence of random variables $W_n$ with distribution $\bW_n$ satisfies an LDP with speed $n$ and a rate function $\rate_W$ satisfying $\rate_W(x)\neq 0$ for all $x\neq 0$. Then the same conclusions as in Theorem \ref{thm:CriticalScale} hold.
\end{thm}

\subsection{Probabilistic representation}

The proof of Theorem \ref{thm:CriticalScalegen} is based on a suitable probabilistic representation of the random variables $\mathscr{Z}_{n,p}$. For the case that the random vectors $X_n$ are uniformly distributed on $\B_p^n$ such a representation was derived in \cite[Lemma 3.1]{APT2018} and for the general situation considered here it is the content of \cite[Proposition 2.7]{APT2019}. It says that
$$
k_n^{-1/2}\mathscr{Z}_{n,p}={n^{1/p}\over k_n^{1/2}}\|P_{E_n}X_n\|_2 \overset{d}{=} {n^{1/p}\over k_n^{1/2}}{\big(\sum\limits_{i=1}^{k_n}g_i^2\big)^{1/2}\over\big(\sum\limits_{i=1}^{n}g_i^2\big)^{1/2}}{\big(\sum\limits_{i=1}^{n}|Z_i|^2\big)^{1/2}\over\big(\sum\limits_{i=1}^{n}|Z_i|^p+W\big)^{1/p}},
$$
where $g_1,\ldots,g_n$ are independent standard Gaussian random variables and $Z_1,\ldots,Z_n$ are independent $p$-generalized Gaussian random variables, which are also independent of all $g_i$'s. Thus,
\begin{equation}\label{eq:repZ}
{n^{1/p}\over k_n^{1/2}}\|P_{E_n}X_n\|_2 \overset{d}{=} {\big({1\over k_n}\sum\limits_{i=1}^{k_n}g_i^2\big)^{1/2}\over\big({1\over n}\sum\limits_{i=1}^{n}g_i^2\big)^{1/2}}{\big({1\over n}\sum\limits_{i=1}^{n}|Z_i|^2\big)^{1/2}\over\big({1\over n}\sum\limits_{i=1}^{n}|Z_i|^p+{W\over n}\big)^{1/p}}.
\end{equation}

\subsection{Proof of Theorem \ref{thm:CriticalScalegen} for $2\leq p<\infty$}

As a consequence of Cram\'er's theorem (Lemma \ref{lem:cramer}) the sequence of random variables ${1\over k_n}\sum\limits_{i=1}^{k_n}g_i^2$ satisfies an LDP with speed $k_n$ and rate function
$$
\rate_G(x) = \begin{cases}
{x-1\over 2}-{1\over 2}\log x &: x>0\\
+\infty &: \text{otherwise},
\end{cases}
$$
compare with \cite[Lemma 5.4]{APT2018}.
Consequently, by the contraction principle (Lemma \ref{lem:refinement contraction principle}) the sequence of random variables $\big({{M_p(2)}\over k_n}\sum\limits_{i=1}^{k_n}g_i^2\big)^{1/2}$ satisfies an LDP with the same speed and rate function
\begin{align}\label{eq:RateFunctionG}
\rate_1(x) = \rate_G\left(\frac{x^2}{M_p(2)}\right) = \begin{cases}
{\frac{x^2}{M_p(2)}-1\over 2}-\log \frac{x}{\sqrt{M_p(2)}} &: x>0\\
+\infty &: \text{otherwise},
\end{cases}
\end{align}
which is the rate function from the statement of the theorem. In the same way, the sequence $\left(\frac{1}{n}\sum_{i=1}^{n}|Z_i|^p\right)^{1/p}$ satisfies an LDP with speed $n$ and rate function
$$
\rate_2(x)=\begin{cases}
\frac{x^p-1}{p}-\log x &: x>0\cr
+\infty &:\textrm{otherwise}.\cr
\end{cases}
$$
The rate functions were explicitly given by \cite[Lemma 5.4]{APT2018}. However, the authors realised that there was an overlooked misprint in the rate function $\rate_2(x)$ given there, so we explicitly write the rate function here.

What remains to show is that the sequences of random variables $k_n^{-1/2}\mathscr{Z}_{n,p}$ and $\big({{M_p(2)}\over k_n}\sum\limits_{i=1}^{k_n}g_i^2\big)^{1/2}$ are exponentially equivalent (recall Lemma \ref{prop:exponentially equivalent}). For this, we fix $\delta,\varepsilon>0$ and write
\begin{align*}
&\Pro\Bigg[\Big({{M_p(2)}\over k_n}\sum\limits_{i=1}^{k_n}g_i^2\Big)^{1/2}\Bigg|1-{\big({1\over n}\sum\limits_{i=1}^n{\frac{Z_i^2}{M_p(2)}}\big)^{1/2}\over \big({1\over n}\sum\limits_{i=1}^{n}g_i^2\big)^{1/2}\big({1\over n}\sum\limits_{i=1}^{n}|Z_i|^p+{W_n\over n}\big)^{1/p}}\Bigg|>\delta\Bigg]\\
&\leq\Pro\bigg[\Big({{M_p(2)}\over k_n}\sum\limits_{i=1}^{k_n}g_i^2\Big)^{1/2}>{\delta\over\varepsilon}\bigg] + \Pro\Bigg[1-{\big({1\over n}\sum\limits_{i=1}^n{\frac{Z_i^2}{M_p(2)}}\big)^{1/2}\over \big({1\over n}\sum\limits_{i=1}^{n}g_i^2\big)^{1/2}\big({1\over n}\sum\limits_{i=1}^{n}|Z_i|^p+{W_n\over n}\big)^{1/p}}>\varepsilon\Bigg]\\
&\qquad\qquad + \Pro\Bigg[1-{\big({1\over n}\sum\limits_{i=1}^n{\frac{Z_i^2}{M_p(2)}}\big)^{1/2}\over \big({1\over n}\sum\limits_{i=1}^{n}g_i^2\big)^{1/2}\big({1\over n}\sum\limits_{i=1}^{n}|Z_i|^p+{W_n\over n}\big)^{1/p}}<-\varepsilon\Bigg] =: T_1 + T_2 + T_3.
\end{align*}
We further analyse the term $T_2$ and observe that, {for every $\varepsilon_1>0$},
\begin{align*}
T_2 &\leq \Pro\Big[\big({1\over n}\sum_{i=1}^n{\frac{Z_i^2}{M_p(2)}}\big)^{1/2}<(1-\varepsilon)^{1/3}\Big] + \Pro\Big[\big({1\over n}\sum_{i=1}^ng_i^2\big)^{1/2}>(1-\varepsilon)^{-1/3}\Big]\\
&\qquad + \Pro\Big[{1\over n}\sum_{i=1}^n|Z_i|^p>(1-\varepsilon)^{-p/3}-\varepsilon_1\Big] + \Pro\Big[{W_n\over n}>\varepsilon_1\Big].
\end{align*}
By Cram\'er's theorem (see Lemma \ref{lem:cramer}), the sequences of random variables ${1\over n}\sum_{i=1}^n{\frac{Z_i^2}{M_p(2)}}$, ${1\over n}\sum_{i=1}^ng_i^2$, and ${1\over n}\sum_{i=1}^n|Z_i|^p$ satisfy LDPs with speed $n$ and rate functions only vanishing at $1$. 
Taking $\varepsilon_1$ such that $(1-\varepsilon)^{-p/3}-\varepsilon_1>1$ and taking into account that, by assumption also the sequence of random variables $W_n/n$ satisfies an LDP with speed $n$ and a rate function that does not vanish on $(\varepsilon_1,\infty)$ we have that the term $T_2$ decays exponentially with speed $n$ and hence satisfies
$$
\limsup_{n\to\infty}{1\over k_n}\log T_2 = -\infty,
$$
since $k_n=o(n)$ by assumption. The same argument also yields that
$$
\limsup_{n\to\infty}{1\over k_n}\log T_3 = -\infty,
$$
which in turn leads to
\begin{align*}
&\limsup_{n\to\infty}{{M_p(2)}\over k_n}\log\Pro\Bigg[\Big({{M_p(2)}\over k_n}\sum\limits_{i=1}^{k_n}g_i^2\Big)^{1/2}\Bigg|1-{\big({1\over n}\sum\limits_{i=1}^n{\frac{Z_i^2}{M_p(2)}}\big)^{1/2}\over \big({1\over n}\sum\limits_{i=1}^{n}g_i^2\big)^{1/2}\big({1\over n}\sum\limits_{i=1}^{n}|Z_i|^p+{W_n\over n}\big)^{1/p}}\Bigg|>\delta\Bigg]\\
&\leq \limsup_{n\to\infty}{1\over k_n}\log T_1 = -\rate_1\Big({\delta\over\varepsilon}\Big) = -{{\delta^2\over{\varepsilon^2M_p(2)}}+1\over 2}-\log{\delta\over{\varepsilon\sqrt{M_p(2)}}}
\end{align*}
with $\rate_1$ given by \eqref{eq:RateFunctionG}.
As $\varepsilon\to 0$ the last expression tends to $-\infty$ for any $\delta>0$. This shows the desired exponential equivalence. \hfill $\Box$

\subsection{Proof of Theorem \ref{thm:CriticalScalegen} for $1\leq p<2$}

The idea of the proof of Theorem \ref{thm:CriticalScalegen} for $1\leq p<2$ is similar to that for $2\leq p<\infty$: identify the dominating term(s) in the probabilistic representation \eqref{eq:repZ} and show exponential equivalence of $k_n^{-1/2}\|\mathscr{Z}_{n,p}\|_2$ to them. However, the analysis in the case where $1\leq p<2$ is a bit more subtle, since depending on the growth of $k_n$ relative to $n^{p/2}$ different terms can take over the dominating role. Moreover, if $k_n=n^{p/2}$ there is more than one such term.

\subsubsection{The case $k_n=o(n^{p/2})$} This situation is similar to that where $2\leq p<\infty$. 
But now the sequence of random variables $\big({1\over n}\sum_{i=1}^nZ_i^2\big)^{1/2}$ satisfies an LDP with speed $n^{p/2}$ instead of $n$ (see \cite[Equation (5)]{APT2018}). This slowdown of the speed is due to the fact that for $1\leq p<2$ the squares of $p$-generalized Gaussian random variables display heavier tails and do not have finite exponential moments. However, since the subspace dimensions satisfy $k_n=o(n^{p/2})$, the proof still works in this situation.\hfill $\Box$

\subsubsection{The case $k_n=n^{p/2}$} To deal with this case we start by observing that by Lemma \ref{JointRateFunction} the sequence of random vectors
$$
\bigg(\Big({1\over k_n}\sum_{i=1}^{k_n}g_i^2\Big)^{1/2},\Big({1\over n}\sum_{i=1}^{n}Z_i^2\Big)^{1/2}\bigg)
$$
satisfies an LDP on $\R^2$ with speed $n^{p/2}$ and rate function $\rate_{G+Z}(x,y)=\rate_G(x)+\rate_Z(y)$. Here,
$$
\rate_G(x) = \begin{cases}
{x^2-1\over 2}-\log x &: x>0\\
+\infty &: \text{otherwise}
\end{cases} \qquad \text{and}\qquad
\rate_Z(y) = \begin{cases}
{1\over p}(y^2-m_p)^{p/2} &: y>m_p\\
+\infty &: \text{otherwise},
\end{cases}
$$
where the constant $m_p$ is given in Theorem \ref{thm:CriticalScale} and the precise form of the LDPs for the individual sequences follow from \cite[Lemma 5.4 and Equation (5)]{APT2018}. By the contraction principle (Lemma \ref{lem:refinement contraction principle}) this implies that the sequence of random variables $\big({1\over k_n}\sum_{i=1}^{k_n}g_i^2\big)^{1/2}\cdot\big({1\over n}\sum_{i=1}^{n}Z_i^2\big)^{1/2}$ satisfies an LDP with speed $n^{p/2}$ and rate function $\rate$ as given in the statement of Theorem \ref{thm:CriticalScale} (b2). It remains to prove that the sequences of random variables $\big({1\over k_n}\sum_{i=1}^{k_n}g_i^2\big)^{1/2}\cdot\big({1\over n}\sum_{i=1}^{n}Z_i^2\big)^{1/2}$ and $n^{-p/2}\|\mathscr{Z}_{n,p}\|$ are exponentially equivalent (recall Lemma \ref{prop:exponentially equivalent}). To show this, fix $\delta,\varepsilon>0$ and observe that
\begin{align*}
&\Pro\Bigg[\big({1\over k_n}\sum_{i=1}^{k_n}g_i^2\big)^{1/2}\big({1\over n}\sum_{i=1}^{n}Z_i^2\big)^{1/2}\Bigg|1-{1\over\big({1\over n}\sum_{i=1}^ng_i^2\big)^{1/2}\big({1\over n}\sum_{i=1}^n|Z_i|^p+{W_n\over n}\big)^{1/p}}\Bigg|>\delta\Bigg]\\
&\leq \Pro\bigg[\big({1\over k_n}\sum_{i=1}^{k_n}g_i^2\big)^{1/2}\big({1\over n}\sum_{i=1}^{n}Z_i^2\big)^{1/2}>{\delta\over\varepsilon}\bigg] + \Pro\Bigg[1-{1\over\big({1\over n}\sum_{i=1}^ng_i^2\big)^{1/2}\big({1\over n}\sum_{i=1}^n|Z_i|^p+{W_n\over n}\big)^{1/p}}>\varepsilon\Bigg]\\
&\qquad\qquad+\Pro\Bigg[1-{1\over\big({1\over n}\sum_{i=1}^ng_i^2\big)^{1/2}\big({1\over n}\sum_{i=1}^n|Z_i|^p+{W_n\over n}\big)^{1/p}}<-\varepsilon\Bigg] =: T_1+T_2+T_3.
\end{align*}
Now, for $T_2$ we obtain that, for any $\varepsilon_1>0$,
\begin{align*}
T_2 \leq \Pro\Big[\big({1\over n}\sum_{i=1}^ng_i^2\big)^{1/2}>(1-\varepsilon)^{-1/2}\Big]+\Pro\Big[{1\over n}\sum_{i=1}^n|Z_i|^p>(1-\varepsilon)^{-p/2}-\varepsilon_1\Big] + \Pro\Big[{W_n\over n}>\varepsilon_1\Big].
\end{align*}
By Cram\'er's theorem (Lemma \ref{lem:cramer}), the sequences of random variables ${1\over n}\sum_{i=1}^ng_i^2$ and ${1\over n}\sum_{i=1}^n|Z_i|^p$ satisfy an LDP with speed $n$, and, by assumption, $W_n/n$ also satisfies an LDP with speed $n$. Moreover, taking $\varepsilon_1$ small enough so that $(1-\varepsilon)^{-p/2}-\varepsilon_1>1$, the corresponding rate functions do not vanish on $((1-\varepsilon)^{-1/2},\infty)$, $((1-\varepsilon)^{-p/2}-\varepsilon_1,\infty)$, and $(\varepsilon_1,\infty)$. Since a similar argument applies to $T_3$, we conclude that
\begin{align*}
\limsup_{n\to\infty}{1\over n^{p/2}}\log T_2 \leq -\infty\qquad\text{and}\qquad\limsup_{n\to\infty}{1\over n^{p/2}}\log T_3 \leq -\infty.
\end{align*}
Thus,
\begin{align*}
&\limsup_{n\to\infty}{1\over n^{p/2}}\log\Pro\bigg[\big({1\over k_n}\sum_{i=1}^{k_n}g_i^2\big)^{1/2}\big({1\over n}\sum_{i=1}^{n}Z_i^2\big)^{1/2}\Big|1-{1\over\big({1\over n}\sum_{i=1}^ng_i^2\big)^{1/2}\big({1\over n}\sum_{i=1}^n|Z_i|^p+{W_n\over n}\big)^{1/p}}\Big|>\delta\bigg]\\
&\leq \limsup_{n\to\infty}{1\over n^{p/2}}\log T_1 = -\rate\Big({\delta\over\varepsilon}\Big),
\end{align*}
where $\rate$ is the rate function from the statement of Theorem \ref{thm:CriticalScale} (b2).
Taking into account the explicit form of $\rate$, we conclude that $\rate(\delta/\varepsilon)\to\infty$, as $\varepsilon\to 0$. This proves the result for the case that ${k_n}=n^{p/2}$.\hfill $\Box$

\subsubsection{The case $k_n=\omega(n^{p/2})$} In this case we claim that the sequence of random variables $\big({1\over n}\sum_{i=1}^nZ_i^2\big)^{1/2}$ is exponentially equivalent to $k_n^{-1/2}\mathscr{Z}_{n,p}$. In fact, for $\delta,\varepsilon>0$ we can write
\begin{align*}
&\Pro\Bigg[\Big({1\over n}\sum\limits_{i=1}^{n}Z_i^2\Big)^{1/2}\Bigg|1-{\big({1\over k_n}\sum\limits_{i=1}^{k_n}g_i^2\big)^{1/2}\over \big({1\over n}\sum\limits_{i=1}^{n}g_i^2\big)^{1/2}\big({1\over n}\sum\limits_{i=1}^{n}|Z_i|^p+{W_n\over n}\big)^{1/p}}\Bigg|>\delta\Bigg]\\
&\leq\Pro\Big[\Big({1\over n}\sum\limits_{i=1}^{n}Z_i^2\Big)^{1/2}>{\delta\over\varepsilon}\Big] + \Pro\Bigg[1-{\big({1\over k_n}\sum\limits_{i=1}^{k_n}g_i^2\big)^{1/2}\over \big({1\over n}\sum\limits_{i=1}^{n}g_i^2\big)^{1/2}\big({1\over n}\sum\limits_{i=1}^{n}|Z_i|^p+{W_n\over n}\big)^{1/p}}>\varepsilon\Bigg]\\
&\qquad\qquad + \Pro\Bigg[1-{\big({1\over k_n}\sum\limits_{i=1}^{k_n}g_i^2\big)^{1/2}\over \big({1\over n}\sum\limits_{i=1}^{n}g_i^2\big)^{1/2}\big({1\over n}\sum\limits_{i=1}^{n}|Z_i|^p+{W_n\over n}\big)^{1/p}}<-\varepsilon\Bigg] =: T_1 + T_2 + T_3,
\end{align*}
and for $T_2$, we obtain that, for any $\varepsilon_1>0$,
\begin{align*}
T_2 &\leq \Pro\Big[{1\over k_n}\sum_{i=1}^{k_n}g_i^2<(1-\varepsilon)^{2/3}\Big] + \Pro\Big[{1\over n}\sum_{i=1}^ng_i^2>(1-\varepsilon)^{-2/3}\Big]\\
&\qquad + \Pro\Big[{1\over n}\sum_{i=1}^n|Z_i|^p>(1-\varepsilon)^{-p/3}-\varepsilon_1\Big] + \Pro\Big[{W_n\over n}>\varepsilon_1\Big].
\end{align*}
Once again, by Cram\'er's theorem (Lemma \ref{lem:cramer}), the sequences of random variables ${1\over n}\sum_{i=1}^ng_i^2$ and ${1\over n}\sum_{i=1}^n|Z_i|^p$ satisfy an LDP with speed $n$, and, by assumption, $W_n/n$ also satisfies an LDP with speed $n$. Again by Cram\'er's theorem (Lemma \ref{lem:cramer}), the sequence of random variables ${1\over k_n}\sum_{i=1}^{k_n}g_i^2$ satisfies an LDP with speed $k_n$. Also and as already discussed above, taking $\varepsilon_1$ small enough so that $(1-\varepsilon)^{-p/3}-\varepsilon_1>1$ , the corresponding rate functions do not vanish on $(-\infty, (1-\varepsilon)^{2/3}$, $((1-\varepsilon)^{-2/3},\infty)$, $((1-\varepsilon)^{-p/3}-\varepsilon_1,\infty)$, and $(\varepsilon_1,\infty)$. This proves that for suitable constants $c_1,c_2>0$,
\begin{align*}
\limsup_{n\to\infty}{1\over n^{p/2}}\log T_2 \leq -\limsup_{n\to\infty}\Big({c_1k_n\over n^{p/2}}+{c_2n\over n^{p/2}}\Big) = -\infty,
\end{align*}
since $k_n=\omega(n^{p/2})$ by assumption and $1\leq p<2$. Similarly, one has that
$$
\limsup_{n\to\infty}{1\over n^{p/2}}\log T_3 = -\infty
$$
and hence
\begin{align*}
&\limsup_{n\to\infty}{1\over n^{p/2}}\log\Pro\Bigg[\Big({1\over n}\sum\limits_{i=1}^{n}Z_i^2\Big)^{1/2}\Bigg|1-{\big({1\over k_n}\sum\limits_{i=1}^{k_n}g_i^2\big)^{1/2}\over \big({1\over n}\sum\limits_{i=1}^{n}g_i^2\big)^{1/2}\big({1\over n}\sum\limits_{i=1}^{n}|Z_i|^p+{W_n\over n}\big)^{1/p}}\Bigg|>\delta\Bigg]\\
&\leq \limsup_{n\to\infty}{1\over n^{p/2}}\log T_1 = -\rate\Big({\delta\over\varepsilon}\Big),
\end{align*}
where $\rate$ is the rate function of the LDP for the sequence of random variables $\big({1\over n}\sum_{i=1}^nZ_i^2\big)^{1/2}$, which holds at speed $n^{p/2}$. We have that
$$
\rate(x) = \begin{cases}
{1\over p}(x^2-m_p)^{p/2} &: x>\sqrt{m_p}\\
+\infty &: \text{otherwise}
\end{cases}
$$
with the constant $m_p$ as in Theorem \ref{thm:CriticalScale} according to \cite[Equation (5)]{APT2018}. In particular, this shows that $\rate(\delta/\varepsilon)\to\infty$, as $\varepsilon\to 0$, and proves the exponential equivalence of the sequences of random variables $\big({1\over n}\sum_{i=1}^nZ_i^2\big)^{1/2}$ and $k_n^{-1/2}\mathscr{Z}_{n,p}$. The proof of Theorem \ref{thm:CriticalScalegen} is thus complete.\hfill $\Box$

\section{Proof of Theorem \ref{thm:MDP} and its generalization}

Recall the notions and notation introduced before Theorem \ref{thm:MDP} and recall the definition of the probability measures $\bP_{\bW,n,p}$ on $\B_p^n$. The goal of this section is to prove the following result, which contains Theorem \ref{thm:MDP} as a special case.

\begin{thm}[MDP on subcritical scales, general version]\label{thm:MDPgen}
Let $2\leq p<\infty$ and $(k_n)_{n\in\N}$ be a sequence of integers such that $1\leq k_n\leq n$ and $\lim\limits_{n\to\infty}{k_n\over n}=\lambda\in[0,1]$. Assume that $k_n=\omega(1)$, and that the sequence $(t_n)_{n\in\N}$ of positive real numbers satisfies $t_n=\omega(1)$, $t_n=o(\sqrt{k_n})$, {and either $t_n=o(\sqrt{n-k_n})$ or $(n-k_n)=o\big(\frac{\sqrt{k_n}}{t_n}\big)$.} For each $n\in\N$ let $\bW_n$ be a probability measure on $\R$ such that
\begin{equation}\label{eq:ConditionW}
\lim_{n\to\infty}{t_n^{-2}}\log\bW_n\big((\delta t_n\sqrt{n},\infty)\big)=-\infty
\end{equation}
for every $\delta>0$. Further, let $X_n$ be a random point with distribution $\bP_{\bW_n,n,p}$ and $E_{n}$ be a uniformly distributed $k_n$-dimensional random subspace {of $\R^n$}. Assume independence of $X_n$ and $E_{n}$ and put
$$
\mathscr{X}_{n,p}:=n^{1/p}\sqrt{{\Gamma\big({1\over p}\big)\over p^{2/p}\Gamma\big({3\over p}\big)}}\,\|P_{E_n}X_n\|_2-\sqrt{k_n}.
$$
Then the sequence of random variables $t_n^{-1}\mathscr X_{n,p}$ satisfies an MDP on $\R$ with speed $t_n^2$ and rate function $\rate:\R\to[0,\infty)$, $\rate(x)=\alpha_{p,\lambda}x^2$ with $\alpha_{p,\lambda}$ given by \eqref{eq:ConstantAlpha}.
\end{thm}

\subsection{Probabilistic representation}

As for the proof of the central limit theorem in \cite{APT2019} and the large deviations principle in \cite{APT2018}, a suitable probabilistic representation for the target random variable $t_n^{-1}\mathscr X_{n,p}$ in terms of families of independent random variables turns out to be one of the most crucial ingredients also in our proof of the moderate deviations principle. In the context of the MDP, it is in fact the representation used in the proof of the central limit theorem, which is most suitable and which we develop further in the present text. To ease comparison we adopt the notation from \cite{APT2019} and recall from the proof of \cite[Theorem 1.1]{APT2019} that for each $n\in\N$, we have
\begin{align*}
\frac{\mathscr X_{n,p}}{t_n} \,\stackrel{d}{=}\, &\frac{\sqrt{\lambda_n}}{2M_p(2)}\,\xi^{(n)}_{p,2} - \frac{\sqrt{\lambda_n}}{p}\,\xi^{(n)}_{p,p} + \frac{1-\lambda_n}{2}\zeta^{(n)}_1-\frac{\sqrt{\lambda_n(1-\lambda_n)}}{2}\,\zeta^{(n)}_2- \sqrt{\lambda_n}\,\frac{W_n}{pt_n\sqrt{n}} \\
&\qquad\qquad\qquad\qquad\qquad\qquad\qquad\qquad\qquad\qquad + \frac{\sqrt{\lambda_n}}{t_n\sqrt{n}}\, \Psi_p\left(\frac{t_n\xi^{(n)}_{p,2}}{ \sqrt{n}},\frac{t_n\xi^{(n)}_{p,p}}{\sqrt{n}},\frac{t_n\zeta^{(n)}_1}{\sqrt{k_n}},\frac{t_n\zeta^{(n)}_3}{\sqrt{n}},\frac{W_n}{n}\right),
\end{align*}
where the random variable $W_n$ has distribution $\bW_n$,
$$\begin{array}{ll}
\xi^{(n)}_{p,2}:=\frac{1}{t_n\sqrt{n}}\sum\limits_{i=1}^n\big(Z_i^2-M_p(2)\big),
&\xi^{(n)}_{p,p} :=\frac{1}{t_n\sqrt{n}}\sum\limits_{i=1}^n\big(|Z_i|^p-1\big),\\
\zeta^{(n)}_1 :=\frac{1}{t_n\sqrt{k_n}}\sum\limits_{i=1}^{k_n}\big(g_i^2-1\big),&\zeta^{(n)}_2 :=\frac{1}{t_n\sqrt{n-k_n}}\sum\limits_{i=k_n+1}^n\big(g_i^2-1\big),\\
\zeta_3^{(n)}:=\sqrt{\lambda_n}\,\zeta^{(n)}_1+\sqrt{1-\lambda_n}\,\zeta^{(n)}_2=\frac{1}{t_n\sqrt{n}}\sum_{i=1}^{n}\big(g_i^2-1\big),
\end{array}$$
and where $\Psi_p:\R^5\to\R $ is a function with the property that there are two constants $M,\delta\in(0,\infty)$ such that $|\Psi_p({\bf x})|\leq M\,\Vert {\bf x}\Vert_2^2$ whenever $\Vert {\bf x}\Vert_2<\delta$. 
At this point we shall remind the reader of our convention that $(Z_i)_{i\in\N}$ stands for a sequence of independent $p$-generalized Gaussian random variables and $(g_i)_{i\in\N}$ for a sequence of independent standard Gaussian random variables and that both sequences are assumed to be independent of each other. In addition, it is understood that the random variables $W_n$ are also independent of all $g_i$'s and $Z_i$'s. We also emphasize that a notation involving the symbol $\xi$ refers to random variables involving the $p$-generalized Gaussian random variables $Z_i$ only, while the symbol $\zeta$ refers to random elements only based on the standard Gaussians $g_i$.

Finally, we shall denote by $\mathscr Y_{n,p}$ the random variable given by
$$
\mathscr Y_{n,p}:=\frac{\sqrt{\lambda_n}}{2M_p(2)}\,\xi^{(n)}_{p,2} - \frac{\sqrt{\lambda_n}}{p}\,\xi^{(n)}_{p,p} + \frac{1-\lambda_n}{2}\zeta^{(n)}_1-\frac{\sqrt{\lambda_n(1-\lambda_n)}}{2}\,\zeta^{(n)}_2,\qquad n\in\N.
$$
Moreover, we define
$$
\widetilde{\mathscr{Y}}_{n,p}:=\frac{\sqrt{\lambda_n}}{2M_p(2)}\,\xi^{(n)}_{p,2} - \frac{\sqrt{\lambda_n}}{p}\,\xi^{(n)}_{p,p} + \frac{1-\lambda_n}{2}\zeta^{(n)}_1,\qquad n\in\N.
$$

\subsection{MDP for $p$-generalized Gaussian random variables}

Specializing Lemma \ref{lem:MDPVector} for $d=1$ to normalized sums of $p$-generalized Gaussian random variables leads to the following MDP.

\begin{lemma}[MDP for sums of powers of $p$-generalized Gaussians]\label{lem:MDP Z^2}
Let $2\leq q\leq  p < \infty$ and $(Z_i)_{i\in\N}$ be a sequence of independent copies of a $p$-generalized Gaussian random variable. Let $(t_n)_{n\in\N}$ be a sequence of positive real numbers such that {$t_n=\omega(1)$} and $t_n=o(\sqrt{n})$. Then the sequence of random variables
\[
\xi_{p,q}^{(n)}:=\frac{1}{t_n\sqrt{n}}\sum_{i=1}^n\big(|Z_i|^q-M_p(q)\big),\qquad n\in\N,
\]
satisfies an MDP on $\R$ with speed $t_n^2$ and rate function
\[
\mathscr I:\R\to[0,\infty),\qquad \mathscr I(x)= \frac{x^2}{2\Var[|Z_1|^q]}={x^2\over 2(M_p(2q)-M_p(q)^2)}.
\]
\end{lemma}
\begin{proof}
Clearly, $\E[|Z_1|^q-M_p(q)]=M_p(q)-M_p(q)=0$ and, using the notation of Lemma \ref{lem:MDPVector}, $\bC^{-1}=1/\Var[|Z_1|^q]$ (in fact, $\bC$ is a $(1\times 1)$-matrix in our case). In addition, we have that $\Var[|Z_1|^q]=M_p(2q)-M_p(q)^2$ from \eqref{eq:Moments}. The result is now a direct consequence of Lemma \ref{lem:MDPVector}, since exponential moments exists by our assumption that $2\leq q\leq p$.

\end{proof}

\begin{rmk}\rm 
It is not strictly necessary to assume that $2\leq q$ in the previous lemma. However, we imposed this condition as we will apply the result for $q=2$, and it shows the reason why we need the condition $p\geq 2$ in Theorem \ref{thm:MDPgen}.
\end{rmk}
%

\subsection{MDP for Gaussian random variables}

As a consequence of Lemma \ref{lem:MDP Z^2} obtained in the previous section, applied with $p=q=2$, we obtain the following MDP for sums of random vectors composed by Gaussian random variables.

\begin{cor}[MDP for sums of Gaussians]\label{cor:MDPGaussian}
Let $(g_i)_{i\in\N}$ be a sequence of standard Gaussian random variables. Then we have that if $k_n=\omega(1)$, and $(t_n)_{n\in\N}$ is a sequence of positive real numbers such that $t_n=\omega(1)$ and $t_n=o(\sqrt{k_n})$, the sequence of random variables $\zeta_1^{(n)}:=\frac{1}{t_n\sqrt{k_n}}\sum_{i=1}^{k_n}(g_i^2-1)$ follows an MDP on $\R$ with speed $t_n^2$ and rate function $\mathscr I(x)=\frac{x^2}{4}$, $x\in\R$ and, if $n-k_n=\omega(1)$ and $(t_n)_{n\in\N}$ is a sequence of positive real numbers such that $t_n=\omega(1)$ and $t_n=o(\sqrt{n-k_n})$, the sequence of random variables $\zeta_2^{(n)}:=\frac{1}{t_n\sqrt{n-k_n}}\sum_{i={k_n+1}}^{n}(g_i^2-1)$ follows an MDP on $\R$ with speed $t_n^2$ and rate function $\mathscr I(x)=\frac{x^2}{4}$, $x\in\R$. Consequently, if $k_n=\omega(1)$, $n-k_n=\omega(1)$, {$t_n=\omega(1)$}, $t_n=o(\sqrt{k_n})$, and  $t_n=o(\sqrt{n-k_n})$, the sequence of random vectors $\zeta^{(n)}:=(\zeta_1^{(n)},\zeta_2^{(n)})$ satisfies an MDP on $\R^2$ with speed $t_n^2$ and rate function $\mathscr I(x)=\frac{\| x\|_2^2}{4}$, $x\in\R^2$.
\end{cor}

\subsection{MDP for correlated vectors of $p$-generalized Gaussian random variables}

In a next step we prove an MDP for a sum of random vectors, whose entries are correlated and composed of $p$-generalized Gaussian random variables.

\begin{lemma}[Bivariate MDP for $p$-generalized Gaussians]\label{lem:MDPVectorPGeneralizedGaussian}
Let $2\leq p < \infty$ and $(Z_i)_{i\in\N}$ be a sequence of independent copies of a $p$-generalized Gaussian random variable. Assume that $(t_n)_{n\in\N}$ is a sequence of positive real numbers such that $t_n=\omega(1)$ and $t_n=o(\sqrt{n})$. Then the sequence of random vectors in $\R^2$ given by
$$
\big(\zeta_1^{(n)},\zeta_2^{(n)}\big)=\frac{1}{t_n\sqrt{n}}\sum_{i=1}^n\big(Z_i^2-M_p(2),|Z_i|^p-1\big)
$$
satisfies an MDP on $\R^2$ with speed $t_n^2$ and rate function
$$
J(x_1,x_2)={1\over 2A_p}\Bigg(px_1^2+p^{4/p}\bigg({\Gamma\big({5\over p}\big)\over\Gamma\big({1\over p}\big)}-{\Gamma\big({3\over p}\big)^2\over\Gamma\big({1\over p}\big)^2}\bigg)x_2^2-4p^{2/p}{\Gamma\big({3\over p}\big)\over\Gamma\big({1\over p}\big)}x_1x_2\Bigg),\quad(x_1,x_2)\in\R^2,
$$
with the constant $A_p$ given by
$$
A_p:=p^{4/p}\bigg({\Gamma\big({5\over p}\big)\over\Gamma\big(1+{1\over p}\big)}-(p+4){\Gamma\big({3\over p}\big)^2\over\Gamma\big({1\over p}\big)^2}\bigg).
$$
\end{lemma}
\begin{proof}
We start by noting that the random vector $\big(|Z_1|^2-M_p(2),|Z_1|^p-1\big)$ is centered. Its $2\times 2$ covariance matrix $\bC$ is given by
$$
\bC=\begin{pmatrix}
c_{11} & c_{12} \\ c_{12} & c_{22}
\end{pmatrix},
$$
where
\begin{align*}
c_{11} &= \Cov\big(|Z_1|^2-M_p(2),|Z_1|^2-M_p(2)\big) = \Var|Z_1|^2 = M_p(4) - M_p(2)^2,\\
c_{12} &= \Cov\big(|Z_1|^2-M_p(2),|Z_1|^p-1\big) = M_p(p+2)-M_p(2)M_p(p) = M_p(p+2)-M_p(2),\\
c_{22} &= \Cov\big(|Z_1|^p-1,|Z_1|^p-1\big) = M_p(2p) - M_p(p)^2 = p,
\end{align*}
recall \eqref{eq:Moments}. Consequently, $\det\bC = p(M_p(4) - M_p(2)^2) - (M_p(p+2)-M_p(2))^2$
and
\begin{align*}
\bC^{-1} = {1\over\det\bC}\begin{pmatrix}
p & M_p(2)-M_p(p+2) \\ M_p(2)-M_p(p+2) & M_p(4) - M_p(2)^2
\end{pmatrix}.
\end{align*}
Expressing this in terms of gamma functions, using \eqref{eq:Mpq}, we arrive at $\det\bC=A_p$ and
\begin{align*}
\bC^{-1} = {1\over A_p}\begin{pmatrix}
p & -2p^{2/p}{\Gamma({3\over p})\over\Gamma({1\over p})} \\ -2p^{2/p}{\Gamma({3\over p})\over\Gamma({1\over p})} & p^{4/p}\bigg({\Gamma({5\over p})\over\Gamma({1\over p})}-{\Gamma({3\over p})^2\over\Gamma({1\over p})^2}\bigg)
\end{pmatrix}.
\end{align*}
Since by assumption $p\geq 2$, exponential moments exist and we can apply Lemma \ref{lem:MDPVector}.
Thus, we conclude that the sequence of pairs $\big(\big(\zeta_1^{(n)},\zeta_2^{(n)}\big)\big)_{n\in\N}$ satisfies an MDP on $\R^2$ with speed $t_n^2$ and rate function
$$
J(x_1,x_2) = {1\over 2}\Big\langle(x_1,x_2)^T,\bC^{-1}(x_1,x_2)^T\Big\rangle,\qquad (x_1,x_2)\in\R^2.
$$
Using the explicit form of $\bC^{-1}$, the result follows.
\end{proof}

\subsection{Intermediate MDP}

Recall the definition of the random variables $\xi_{p,2}^{(n)}$, $\xi_{p,p}^{(n)}$, $\zeta_1^{(n)}$, and $\zeta_2^{(n)}$. Using what has been proved in the previous section, we arrive at the following MDP for the sequence of random vectors $(\xi_{p,2}^{(n)},\xi_{p,p}^{(n)},\zeta_1^{(n)})$ and $(\xi_{p,2}^{(n)},\xi_{p,p}^{(n)},\zeta_1^{(n)},\zeta_2^{(n)})$.

\begin{lemma}[Multivariate MDP for the core term]\label{lem:MDPVectorR4}
The following sequences of random vectors satisfy MDPs with speed $t_n^2$ and the given rate functions:
\begin{itemize}
	\item[(a)] If $k_n=\omega(1)$, $t_n=\omega(1)$, and $t_n=o(\sqrt{k_n})$, then the sequence of random vectors $(\xi_{p,2}^{(n)},\xi_{p,p}^{(n)},\zeta_1^{(n)})$ satisfies an MDP on $\R^3$ with speed $t_n^2$ and rate function
	$$
	I(x_1,x_2,x_3)=J(x_1,x_2)+\frac{x_3^2}{4},
	$$
	where $J(\,\cdot\,,\,\cdot\,)$ is the rate function from Lemma \ref{lem:MDPVectorPGeneralizedGaussian}.
	\item[(b)] If $k_n=\omega(1)$, $n-k_n=\omega(1)$, $t_n=\omega(1)$, $t_n=o(\sqrt{k_n})$, and $t_n=o(\sqrt{n-k_n})$, then the sequence of random vectors $(\xi_{p,2}^{(n)},\xi_{p,p}^{(n)},\zeta_1^{(n)},\zeta_2^{(n)})$ satisfies an MDP on $\R^4$ with speed $t_n^2$ and rate function
	$$
	I(x_1,x_2,x_3,x_4)=J(x_1,x_2)+\frac{x_3^2+x_4^2}{4},
	$$
	where $J(\,\cdot\,,\,\cdot\,)$ is the rate function from Lemma \ref{lem:MDPVectorPGeneralizedGaussian}.
\end{itemize}
\end{lemma}
\begin{proof}
We only prove part (b), since the proof of part (a) is completely analogous and simpler. We note that $\xi_{p,2}^{(n)}$ and $\xi_{p,p}^{(n)}$ are independent from $\zeta_1^{(n)}$ and $\zeta_2^{(n)}$. We have seen in Lemma \ref{lem:MDPVectorPGeneralizedGaussian} that the sequence of random vectors $(\xi_{p,2}^{(n)},\xi_{p,p}^{(n)})$ satisfies an MDP on $\R^2$ with speed $t_n^2$ and rate function $J(x_1,x_2)$. Also, Corollary \ref{cor:MDPGaussian} implies that the sequence of random vectors $(\zeta_1^{(n)},\zeta_2^{(n)})$ satisfies an MDP on $\R^2$ with speed $t_n^2$ and rate function $L(x_3,x_4)={x_3^2+x_4^2\over 4}$. Thus, the sequence of random vectors $(\xi_{p,2}^{(n)},\xi_{p,p}^{(n)},\zeta_1^{(n)},\zeta_2^{(n)})$ satisfies an MDP on $\R^4$ with the same speed $t_n^2$, whose rate function is the sum of the rate functions $J(x_1,x_2)$ and $L(x_3,x_4)$, see Lemma \ref{JointRateFunction}.
\end{proof}

Next, we shall apply the contraction principle from Lemma \ref{lem:refinement contraction principle} to deduce the following intermediate MDP for the random sequence $(\mathscr Y_{n,p})_{n\in\N}$.

\begin{lemma}[Intermediate MDP for the core term]\label{lem:IntermediateMDP}
The following sequences of random variables satisfy MDPs with speed $t_n^2$ and the given rate functions:
\begin{itemize}
	\item[(a)] If $k_n=\omega(1)$, $t_n=\omega(1)$, and $t_n=o(\sqrt{k_n})$, then the sequence of random variables $(\widetilde{\mathscr{Y}}_{n,p})_{n\in\N}$ satisfies an MDP on $\R$ with speed $t_n^2$ and rate function
	$$
	I(y)=\inf\Big\{J(x_1,x_2)+{x_3^2\over 4}:\widetilde{F}(x_1,x_2,x_3)=y\Big\},\qquad y\in\R,
	$$
	with the function $\widetilde{F}:\R^3\to\R$ given by
	$$
	F(x_1,x_2,x_3):=\frac{\sqrt{\lambda}}{2M_p(2)}x_1-\frac{\sqrt{\lambda}}{p}x_2+\frac{1-\lambda}{2}x_3.
	$$
	\item[(b)] If $k_n=\omega(1)$, $n-k_n=\omega(1)$, $t_n=\omega(1)$, $t_n=o(\sqrt{k_n})$, and $t_n=o(\sqrt{n-k_n})$, then the sequence of random variables $(\mathscr Y_{n,p})_{n\in\N}$ satisfies an MDP on $\R$ with speed $t_n^2$ and rate function
	$$
	I(y)=\inf\Big\{J(x_1,x_2)+{x_3^2+x_4^2\over 4}:F(x_1,x_2,x_3,x_4)=y\Big\},\qquad y\in\R,
	$$
	with the function $F:\R^4\to\R$ given by
	$$
	F(x_1,x_2,x_3,x_4):=\frac{\sqrt{\lambda}}{2M_p(2)}x_1-\frac{\sqrt{\lambda}}{p}x_2+\frac{1-\lambda}{2}x_3-\frac{\sqrt{\lambda(1-\lambda)}}{2}x_4.
	$$
\end{itemize}
\end{lemma}
\begin{proof}
Again, we only prove part (b), since the proof of part (a) is completely analogous and simpler. For each $n\in\N$ let $F_n:\R^4\to\R$ be the function
$$
F_n(x_1,x_2,x_3,x_4):=\frac{\sqrt{\lambda_n}}{2M_p(2)}x_1-\frac{\sqrt{\lambda_n}}{p}x_2+\frac{1-\lambda_n}{2}x_3-\frac{\sqrt{\lambda_n(1-\lambda_n)}}{2}x_4
$$
and observe that $F(x_1,x_2,x_3,x_4)$ is the pointwise limit of $F_n(x_1,x_2,x_3,x_4)$, as $n\to\infty$. Note that for each $n\in\N$, $\mathscr Y_{n,p}$ has the same distribution as $F_n(\xi_{p,2}^{(n)},\xi_{p,p}^{(n)},\zeta_1^{(n)},\zeta_2^{(n)})$.

Our goal is to apply Lemma \ref{lem:refinement contraction principle} to conclude from Lemma \ref{lem:MDPVectorR4} the MDP for the random sequence $(\mathscr Y_{n,p})_{n\in\N}$. For this we need to argue that
\begin{equation}\label{eq:IntermediateMDPComputation1}
\lim_{n\to\infty}{1\over t_n^2}\log\Pro\big[(\xi_{p,2}^{(n)},\xi_{p,p}^{(n)},\zeta_1^{(n)},\zeta_2^{(n)})\in\Gamma_{n,\delta}\big]=-\infty\,,
\end{equation}
where for $\delta>0$, $\Gamma_{n,\delta}$ is the set $\Gamma_{n,\delta}:=\{x\in\R^{4}:|F_n(x)-F(x)|>\delta\}$. To prove this, we first note that
\begin{align*}
\Pro\big[(\xi_{p,2}^{(n)},\xi_{p,p}^{(n)},\zeta_1^{(n)},\zeta_2^{(n)})\in  \Gamma_{n,\delta}\big]&\leq\Pro\left[\frac{|\sqrt{\lambda_n}-\sqrt{\lambda}|}{t_n\sqrt{n}}\left|\sum_{i=1}^n \big(Z_i^2-M_p(2)\big)\right|>\frac{\delta M_p(2)}{2}\right]\cr
&\quad+\Pro\left[\frac{|\sqrt{\lambda_n}-\sqrt{\lambda}|}{t_n\sqrt{n}}\left|\sum_{i=1}^n \big(|Z_i|^p-1\big)\right|>\frac{p\delta }{4}\right]+\Pro\left[\frac{|\lambda_n-\lambda|}{t_n\sqrt{k_n}}\left|\sum_{i=1}^{k_n} \big(g_i^2-1\big)\right|>\frac{\delta }{2}\right]\cr
&\quad+\Pro\left[\frac{\sqrt{\lambda_n(1-\lambda_n)}-\sqrt{\lambda(1-\lambda)|}}{t_n\sqrt{n-k_n}}\left|\sum_{i=k_n+1}^n \big(g_i^2-1\big)\right|>\frac{\delta}{2}\right].
\end{align*}
If there exists a constant $c_1:=c_1(\delta,p)\in(0,\infty)$ such that $\frac{t_n\sqrt{n}}{|\sqrt{\lambda_n}-\sqrt{\lambda}|}\geq c_1n$, then, by Cram\'er's theorem in Lemma \ref{lem:cramer} the sequence $\frac{1}{n}\sum_{i=1}^n \big(Z_i^2-M_p(2)\big)$ satisfies an LDP with speed $n$ and some rate function. So, there exists a constant $c_2:=c_2(\delta,p)\in(0,\infty)$ only depending on $\delta$ and on $p$ such that
\begin{align*}
&\Pro\left[\frac{|\sqrt{\lambda_n}-\sqrt{\lambda}|}{t_n\sqrt{n}}\left|\sum_{i=1}^n \big(Z_i^2-M_p(2)\big)\right|>\frac{\delta M_p(2)}{2}\right]\leq\Pro\left[\frac{1}{n}\left|\sum_{i=1}^n \big(Z_i^2-M_p(2)\big)\right|>\frac{c_1\delta M_p(2)}{2}\right]\leq e^{-c_2n}
\end{align*}
so that
$$
\lim_{n\to\infty}\frac{1}{t_n^2}\log\Pro\left[\frac{|\sqrt{\lambda_n}-\sqrt{\lambda}|}{t_n\sqrt{n}}\left|\sum_{i=1}^n \big(Z_i^2-M_p(2)\big)\right|>\frac{\delta M_p(2)}{2}\right]\leq-\lim_{n\to\infty}\frac{c_2n}{t_n^2}=-\infty
$$
by our assumption on the growth of $t_n$. On the contrary, if such constant does not exist, then by Lemma \ref{lem:MDP Z^2} there exists another constant $c_3:=c_3(\delta,p)\in(0,\infty)$ depending on $\delta$ and $p$ only such that
$$
\Pro\left[\frac{|\sqrt{\lambda_n}-\sqrt{\lambda}|}{t_n\sqrt{n}}\left|\sum_{i=1}^n \big(Z_i^2-M_p(2)\big)\right|>\frac{\delta M_p(2)}{2}\right]\leq e^{-\frac{c_3t_n^2}{(\sqrt{\lambda_n}-\sqrt{\lambda})^2}},
$$
which also implies that
\begin{align*}
&\lim_{n\to\infty}\frac{1}{t_n^2}\log\Pro\left[\frac{|\sqrt{\lambda_n}-\sqrt{\lambda}|}{t_n\sqrt{n}}\left|\sum_{i=1}^n \big(Z_i^2-M_p(2)\big)\right|>\frac{\delta M_p(2)}{2}\right]\leq-\lim_{n\to\infty}\frac{c_3}{(\sqrt{\lambda_n}-\sqrt{\lambda})^2}=-\infty,
\end{align*}
since $\lambda_n\to\lambda$, as $n\to\infty$. Thus, in any case we have that
$$
\lim_{n\to\infty}\frac{1}{t_n^2}\log\Pro\left[\frac{|\sqrt{\lambda_n}-\sqrt{\lambda}|}{t_n\sqrt{n}}\left|\sum_{i=1}^n Z_i^2-M_p(2)\right|>\frac{\delta M_p(2)}{2}\right]=-\infty.
$$
The rest of the terms are treated in the same way, which eventually proves \eqref{eq:IntermediateMDPComputation1}.
\end{proof}

\subsection{Explicit form of the rate function}

After having proved an intermediate MDP in Lemma \ref{lem:IntermediateMDP}, we shall now provide an explicit form for the rate function $I(y)$.

\begin{lemma}[Explicit rate function]\label{lem:RateFunction}
The rate functions $I:\R\to[0,\infty]$ in Lemma \ref{lem:IntermediateMDP} (a) and (b) are given by $I(y)=\alpha_{p,\lambda}y^2$ with
$$
\alpha_{p,\lambda}
{2\Gamma({1\over p})\Gamma({3\over p})^2\over (2p-\lambda(4+3p))\Gamma(1+{1\over p})\Gamma({3\over p})^2+\lambda\Gamma({1\over p})^2\Gamma({5\over p})} .
$$
\end{lemma}
\begin{proof}
For parameters $A,a,b,c\in\R$ let us introduce the function $\widetilde{G}:\R^4\to\R$ by
$$
\widetilde{G}(x_1,x_2,x_3,x_4) := {1\over 2A}\big(ax_1^2+bx_2^2+cx_1x_2\big)+{x_3^2+x_4^2\over 4}.
$$
Our goal is to minimize this function subject to the condition that $\alpha x_1+\beta x_2+\gamma x_3+\delta x_4=y$ for fixed $\alpha,\beta,\gamma,\delta\in\R$ and given $y\in\R$. For that purpose we define $G:\R^5\to\R$ by
$$
G(x_1,x_2,x_3,x_4,\lambda) := \widetilde{G}(x_1,x_2,x_3,x_4)+\lambda\big(\alpha x_1+\beta x_2+\gamma x_3+\delta x_4-y\big).
$$
Differentiation with respect to $x_1,\ldots,x_4$ gives
\begin{align*}
{\partial G\over\partial x_1} &= {1\over 2A}\big(2ax_1+cx_2\big)+\alpha\lambda,\qquad {\partial G\over\partial x_2} = {1\over 2A}\big(2bx_2+cx_1\big)+\beta\lambda,\\
{\partial G\over\partial x_3} &= {x_3\over 2}+\gamma\lambda,\hspace{2.8cm}{\partial G\over\partial x_4} = {x_4\over 2}+\delta\lambda.
\end{align*}
The system ${\partial G\over\partial x_1}=0,{\partial G\over\partial x_2}=0,{\partial G\over\partial x_3}=0,{\partial G\over\partial x_4}=0$ has the unique solution
\begin{align*}
x_1^0 = {2A\lambda(\beta c-2\alpha b)\over 4ab-c^2},\qquad x_2^0={2A\lambda(\alpha c-2\beta a)\over 4ab-c^2},\qquad x_3^0=-2\gamma\lambda,\qquad x_4^0=-2\delta\lambda.
\end{align*}
Plugging this into ${\partial G\over\partial\lambda}=\alpha x_1+\beta x_2+\gamma x_3+\delta x_4-y=0$ and solving for $\lambda$ yields the unique solution
$$
\lambda = y\,{c^2-4ab\over 4A(\alpha^2b+\beta^2a-\alpha\beta c)+2(4ab-c^2)(\delta^2+\gamma^2)}.
$$
Plugging this in turn into the expressions for $x_1^0,x_2^0,x_3^0,x_4^0$ obtained above yields
\begin{align*}
x_1^0 &= y\,{A(2\alpha b-\beta c)\over 2A(\alpha^2b+\beta^2a-\alpha\beta c)+2(4ab-c^2)(\delta^2+\gamma^2)},\\
x_2^0 &= y\,{A(2\beta a-\alpha c)\over 2A(\alpha^2b+\beta^2a-\alpha\beta c)+2(4ab-c^2)(\delta^2+\gamma^2)},\\
x_3^0 &= y\,{\gamma(4ab-c^2)\over 2A(\alpha^2b+\beta^2a-\alpha\beta c)+2(4ab-c^2)(\delta^2+\gamma^2)},\\
x_4^0 &= y\,{\delta(4ab-c^2)\over 2A(\alpha^2b+\beta^2a-\alpha\beta c)+2(4ab-c^2)(\delta^2+\gamma^2)}.
\end{align*}
Consequently,
\begin{align*}
\widetilde{G}(x_1^0,x_2^0,x_3^0,x_4^0) = y^2\,{4ab-c^2\over 4\big(2A(\alpha^2b+\beta^2a-\alpha\beta c)+(4ab-c^2)(\delta^2+\gamma^2)\big)}.
\end{align*}
In our set up, the parameters $A,a,b,c$ and $\alpha,\beta,\gamma,\delta$ are given as follows:
\begin{align*}
A=A_p,\qquad a=p,\qquad b=p^{4/p}\bigg({\Gamma({5\over p})\over\Gamma({1\over p})}-{\Gamma({3\over p})^2\over\Gamma({1\over p})^2}\bigg),\qquad c=-4p^{2/p}{\Gamma({3\over p})\over\Gamma({1\over p})}
\end{align*}
and
\begin{align*}
\alpha={\sqrt{\lambda}\over 2M_p(2)},\qquad\beta={\sqrt{\lambda}\over p},\qquad\gamma={1-\lambda\over 2},\qquad\delta=\sqrt{\lambda(1-\lambda)\over 2}.
\end{align*}
Plugging this into the above expression leads to
$$
{4ab-c^2\over 4\big(2A(\alpha^2b+\beta^2a-\alpha\beta c)+(4ab-c^2)(\delta^2+\gamma^2)\big)} = {2\Gamma({1\over p})\Gamma({3\over p})^2\over (2p-\lambda(4+3p))\Gamma(1+{1\over p})\Gamma({3\over p})^2+\lambda\Gamma({1\over p})^2\Gamma({5\over p})} = \alpha_{p,\lambda},
$$
after simplifications, where $\alpha_{p,\lambda}$ was defined in \eqref{eq:ConstantAlpha}. Note that this covers both cases $\lim_{n\to\infty}(n-k_n)=\infty$ and $\lim_{n\to\infty}(n-k_n)<\infty$, since in the latter case we automatically have $\lambda=1$, which in turn corresponds to the choice $\gamma=\delta=0$. This completes the argument.
\end{proof}

\subsection{MDP for $\mathscr X_{n,p}$ via exponential equivalence}

The purpose of this section is to complete the proof of the MDP for $(\mathscr X_{n,p})_{n\in\N}$ in Theorem \ref{thm:MDPgen}. We do this by showing that this sequence is exponentially equivalent to the random sequence $(\mathscr Y_{n,p})_{n\in\N}$. Moreover, if {$(n-k_n)=o\big(\frac{\sqrt{k_n}}{t_n}\big)$} we shall argue in addition that the random sequences $(\widetilde{\mathscr{Y}}_{n,p})_{n\in\N}$ and $(\mathscr Y_{n,p})_{n\in\N}$ are exponentially equivalent as well. As a consequence, $(\mathscr X_{n,p})_{n\in\N}$ and $(\mathscr Y_{n,p})_{n\in\N}$ satisfy an MDP at the same speed and with the same rate function (see Lemma \ref{prop:exponentially equivalent}), which implies that Theorem \ref{thm:MDPgen} follows from the intermediate MDP in Lemma \ref{lem:IntermediateMDP}. Moreover, the explicit form of the rate function is a consequence of Lemma \ref{lem:RateFunction}. So, what remains to prove is the following exponential equivalence.

\begin{lemma}[Exponential equivalence]
If $k_n=\omega(1)$, $t_n=\omega(1)$, and $t_n=o(\sqrt{k_n})$ the sequence of random variables $t_n^{-1}\mathscr X_{n,p}$ and the sequence of random variables $\mathscr Y_{n,p}$ are exponentially equivalent. Moreover, {if also $(n-k_n)=o\big(\frac{\sqrt{k_n}}{t_n}\big)$}, then the sequences of random variables $(\widetilde{\mathscr{Y}}_{n,p})_{n\in\N}$ and $(\mathscr Y_{n,p})_{n\in\N}$ are exponentially equivalent.
\end{lemma}
\begin{proof}
We start by noting that, for any $\varepsilon>0$,
\begin{align*}
&\frac{1}{t_n^2}\log\Pro\left[\Big|\frac{1}{t_n}\mathscr X_{n,p}-\mathscr Y_{n,p}\Big|>\varepsilon\right]\\
&\leq\frac{1}{t_n^2}\log\left(\Pro\left[\frac{\sqrt{\lambda_n}}{pt_n\sqrt{n}}|W_n|>\frac{\epsilon}{2}\right]+\Pro\left[\frac{\sqrt{\lambda_n}}{t_n\sqrt{n}}\left|\Psi_p\left(\frac{t_n\xi^{(n)}_{p,2}}{ \sqrt{n}},\frac{t_n\xi^{(n)}_{p,p}}{\sqrt{n}},\frac{t_n\zeta^{(n)}_1}{\sqrt{k_n}},\frac{t_n\zeta_3^{(n)}}{\sqrt{n}},\frac{W_n}{n}\right)\right|>\frac{\epsilon}{2}\right]\right)\cr
&\leq\frac{1}{t_n^2}\log\left(\Pro\left[\frac{\sqrt{\lambda_n}}{pt_n\sqrt{n}}|W_n|>\frac{\epsilon}{2}\right]+\Pro\left[\frac{\sqrt{\lambda_n}}{t_n\sqrt{n}}\left\Vert\left(\frac{t_n\xi^{(n)}_{p,2}}{ \sqrt{n}},\frac{t_n\xi^{(n)}_{p,p}}{\sqrt{n}},\frac{t_n\zeta^{(n)}_1}{\sqrt{k_n}},\frac{t_n\zeta_3^{(n)}}{\sqrt{n}},\frac{W_n}{n}\right)\right\Vert_2^2>\frac{\varepsilon}{2M}\right]\right.\cr
&\qquad\qquad\left.+\Pro\left[\frac{\left|\Psi_p\left(\frac{t_n\xi^{(n)}_{p,2}}{ \sqrt{n}},\frac{t_n\xi^{(n)}_{p,p}}{\sqrt{n}},\frac{t_n\zeta^{(n)}_1}{\sqrt{k_n}},\frac{t_n\zeta_3^{(n)}}{\sqrt{n}},\frac{W_n}{n}\right)\right|}{\left\Vert\left(\frac{t_n\xi^{(n)}_{p,2}}{ \sqrt{n}},\frac{t_n\xi^{(n)}_{p,p}}{\sqrt{n}},\frac{t_n\zeta^{(n)}_1}{\sqrt{k_n}},\frac{t_n\zeta_3^{(n)}}{\sqrt{n}},\frac{W_n}{n}\right)\right\Vert_2^2}>M\right]\right),
\end{align*}
where $M$ is the constant associated to $\Psi_p$. On the one hand, by assumption \eqref{eq:ConditionW}, we have that
$$
\frac{1}{t_n^2}\log\Pro\left[\frac{\sqrt{\lambda_n}}{pt_n\sqrt{n}}|W_n|>\frac{\epsilon}{2}\right]\leq\frac{1}{t_n^2}\log\Pro\left[|W_n|>\frac{p\epsilon t_n\sqrt{n}}{2}\right]\to-\infty,
$$
as $n\to\infty$. On the other hand, notice that, if $M$ and $\delta$ are the constants associated to $\Psi_p$,
\begin{align*}
&\Pro\left[\frac{\left|\Psi_p\left(\frac{t_n\xi^{(n)}_{p,2}}{ \sqrt{n}},\frac{t_n\xi^{(n)}_{p,p}}{\sqrt{n}},\frac{t_n\zeta^{(n)}_1}{\sqrt{k_n}},\frac{t_n\zeta_3^{(n)}}{\sqrt{n}},\frac{W_n}{n}\right)\right|}{\left\Vert\left(\frac{t_n\xi^{(n)}_{p,2}}{ \sqrt{n}},\frac{t_n\xi^{(n)}_{p,p}}{\sqrt{n}},\frac{t_n\zeta^{(n)}_1}{\sqrt{k_n}},\frac{t_n\zeta_3^{(n)}}{\sqrt{n}},\frac{W_n}{n}\right)\right\Vert_2^2}>M\right]\\
&\qquad\leq\Pro\left[\left\Vert\left(\frac{t_n\xi^{(n)}_{p,2}}{ \sqrt{n}},\frac{t_n\xi^{(n)}_{p,p}}{\sqrt{n}},\frac{t_n\zeta^{(n)}_1}{\sqrt{k_n}},\frac{t_n\zeta_3^{(n)}}{\sqrt{n}},\frac{W_n}{n}\right)\right\Vert_2>\delta\right]\cr
&\qquad\qquad+\Pro\left[\frac{t_n|\xi^{(n)}_{p,2}|}{\sqrt{n}}>\frac{\delta}{\sqrt{5}}\right]+\Pro\left[\frac{t_n|\xi^{(n)}_{p,p}|}{\sqrt{n}}>\frac{\delta}{\sqrt{5}}\right]+\Pro\left[\frac{t_n|\zeta^{(n)}_1|}{\sqrt{k_n}}>\frac{\delta}{\sqrt{5}}\right]\cr
&\qquad\qquad+\Pro\left[\frac{t_n|\zeta_3^{(n)}|}{\sqrt{n}}>\frac{\delta}{\sqrt{5}}\right]+\Pro\left[\frac{|W_n|}{n}>\frac{\delta}{\sqrt{5}}\right].
\end{align*}
By Cram\'er's theorem in Lemma \ref{lem:cramer}, the sequence of random variables $\frac{t_n|\xi^{(n)}_{p,2}|}{\sqrt{n}}=\frac{1}{n}\sum_{i=1}^n(Z_i^2-M_p(2))$ follows an LDP with speed $n$. So, there exists a constant $c_1:=c_1(p,\delta)\in(0,\infty)$ depending on $\delta$ and $p$ only, and $N_0\in\N$ such that if $n\geq N_0$,
$$
\Pro\left[\frac{t_n|\xi^{(n)}_{p,2}|}{\sqrt{n}}>\frac{\delta}{\sqrt{5}}\right]\leq e^{-c_1n}.
$$
Therefore,
$$
\lim_{n\to\infty}\frac{1}{t_n^2}\log\Pro\left[\frac{t_n|\xi^{(n)}_{p,2}|}{\sqrt{n}}>\frac{\delta}{\sqrt{5}}\right]\leq-\lim_{n\to\infty}\frac{c_1n}{t_n^2}=-\infty.
$$
In the same way we conclude that
$$
\lim_{n\to\infty}\frac{1}{t_n^2}\log\Pro\left[\frac{t_n|\xi^{(n)}_{p,p}|}{\sqrt{n}}>\frac{\delta}{\sqrt{5}}\right]-\infty\qquad\text{and}\qquad \lim_{n\to\infty}\frac{1}{t_n^2}\log\Pro\left[\frac{t_n|\zeta_3^{(n)}|}{\sqrt{n}}>\frac{\delta}{\sqrt{5}}\right]-\infty.
$$
Also, the sequence of random variables $\frac{t_n|\zeta^{(n)}_1|}{\sqrt{k_n}}=\frac{1}{k_n}\sum_{i=1}^{k_n}(g_i^2-1)$ satisfies an LDP with speed $k_n$, implying that there exists a constant $c_2:=c_2(\delta)\in(0,\infty)$ only depending on $\delta$, and $N_1\in\N$ such that if $n\geq N_1$, then
$$
\Pro\left[\frac{t_n|\zeta^{(n)}_1|}{\sqrt{k_n}}>\frac{\delta}{\sqrt{5}}\right]\leq e^{-c_2k_n}.
$$
Therefore, by our assumption on the growth of $k_n$, this implies that
$$
\lim_{n\to\infty}\frac{1}{t_n^2}\log\Pro\left[\frac{t_n|\zeta^{(n)}_1|}{\sqrt{k_n}}>\frac{\delta}{\sqrt{5}}\right]\leq-\lim_{n\to\infty}\frac{c_2k_n}{t_n^2}=-\infty.
$$
Besides, by condition \eqref{eq:ConditionW} 
$$
\lim_{n\to\infty}\frac{1}{t_n^2}\log\Pro\left[\frac{|W_n|}{n}>\frac{\delta}{\sqrt{5}}\right]\leq\lim_{n\to\infty}\frac{1}{t_n^2}\log\Pro\left[|W_n|>\frac{\delta n}{\sqrt{5}}\right]=-\infty.
$$
Finally,
\begin{align*}
&\Pro\left[\frac{\sqrt{\lambda_n}}{t_n\sqrt{n}}\left\Vert\left(\frac{t_n\xi^{(n)}_{p,2}}{ \sqrt{n}},\frac{t_n\xi^{(n)}_{p,p}}{\sqrt{n}},\frac{t_n\zeta^{(n)}_1}{\sqrt{k_n}},\frac{t_n\zeta_3^{(n)}}{\sqrt{n}},\frac{W_n}{n}\right)\right\Vert_2^2>\frac{\varepsilon}{2M}\right]\cr
&\qquad\leq\Pro\left[\frac{t_n\xi^{(n)}_{p,2}}{\sqrt{n}}>\frac{\sqrt{\varepsilon t_n\sqrt{n}}}{\lambda_n^{1/4}\sqrt{10M}}\right]+\Pro\left[\frac{t_n\xi^{(n)}_{p,p}}{\sqrt{n}}>\frac{\sqrt{\varepsilon t_n\sqrt{n}}}{\lambda_n^{1/4}\sqrt{10M}}\right]\cr
&\qquad\qquad+\Pro\left[\frac{t_n\zeta^{(n)}_1}{\sqrt{k_n}}>\frac{\sqrt{\varepsilon t_n\sqrt{n}}}{\lambda_n^{1/4}\sqrt{10M}}\right]+\Pro\left[\frac{t_n\zeta_3^{(n)}}{\sqrt{n}}>\frac{\sqrt{\varepsilon t_n\sqrt{n}}}{\lambda_n^{1/4}\sqrt{10M}}\right]\\
&\qquad\qquad+\Pro\left[\frac{|W_n|}{n}>\frac{\sqrt{\varepsilon t_n\sqrt{n}}}{\lambda_n^{1/4}\sqrt{10 M}}\right].
\end{align*}
Like before, there exists a constant $c_1:=c_1(\varepsilon,M)\in(0,\infty)$ depending only on $\varepsilon$ and $M$, and $N_0\in\N$ such that, for all $n\geq N_0$,
$$
\Pro\left[\frac{t_n\xi^{(n)}_{p,2}}{\sqrt{n}}>\frac{\sqrt{\varepsilon t_n\sqrt{n}}}{\lambda_n^{1/4}\sqrt{10M}}\right]\leq\Pro\left[\frac{t_n\xi_1^{(n)}}{\sqrt{n}}>\frac{\sqrt{\varepsilon}}{\sqrt{10M}}\right]\leq e^{-c_1 n},
$$
which in turn implies that
$$
\lim_{n\to\infty}\frac{1}{t_n^2}\log\Pro\left[\frac{t_n\xi_1^{(n)}}{\sqrt{n}}>\frac{\sqrt{\varepsilon t_n\sqrt{n}}}{\lambda_n^{1/4}\sqrt{10M}}\right]\leq-\lim_{n\to\infty}\frac{c_1 n}{t_n^2}=-\infty.
$$
In the same way, we conclude that
\begin{align*}
& \lim_{n\to\infty}\frac{1}{t_n^2}\log\Pro\left[\frac{t_n\xi^{(n)}_{p,p}}{\sqrt{n}}>\frac{\sqrt{\varepsilon t_n\sqrt{n}}}{\lambda_n^{1/4}\sqrt{10M}}\right]=-\infty,\\
& \lim_{n\to\infty}\frac{1}{t_n^2}\log\Pro\left[\frac{t_n\zeta_3^{(n)}}{\sqrt{n}}>\frac{\sqrt{\varepsilon t_n\sqrt{n}}}{\lambda_n^{1/4}\sqrt{10M}}\right]\leq-\lim_{n\to\infty}\frac{c_1 n}{t_n^2}=-\infty,\\
& \lim_{n\to\infty}\frac{1}{t_n^2}\log\Pro\left[\frac{t_n\zeta^{(n)}_1}{\sqrt{n}}>\frac{\sqrt{\varepsilon t_n\sqrt{n}}}{\lambda_n^{1/4}\sqrt{10M}}\right]\leq-\lim_{n\to\infty}\frac{c_2 k_n}{t_n^2}=-\infty.
\end{align*}
Also, by condition \eqref{eq:ConditionW}, since $t_n=\omega(1)$
$$
\lim_{n\to\infty}\frac{1}{t_n^2}\log\Pro\left[|W_n|>\frac{n\sqrt{\varepsilon t_n\sqrt{n}}}{\lambda_n^{1/4}\sqrt{10 M}}\right]\leq\lim_{n\to\infty}\frac{1}{t_n^2}\log\Pro\left[|W_n|>\frac{n^{5/4}\sqrt{t_n}\sqrt{\varepsilon}}{\sqrt{10 M}}\right]=-\infty.
$$
As a consequence, we have that if  $t_n=o(\sqrt{k_n})$, as $n\to\infty$, for any $\varepsilon>0$,
$$
\lim_{n\to\infty}\frac{1}{t_n^2}\log\Pro\left[\Big|\frac{1}{t_n}\mathscr X_{n,p}-\mathscr Y_{n,p}\Big|>\varepsilon\right]=-\infty.
$$
This proves that the sequences of random variables $\mathscr X_{n,p}$ and $\mathscr Y_{n,p}$ are exponentially equivalent. 

To finish the proof, in view of Lemma \ref{prop:exponentially equivalent} it is enough to argue that if {$(n-k_n)=o\big(\frac{\sqrt{k_n}}{t_n}\big)$}, then the sequences of random variables $\widetilde{\mathscr{Y}}_{n,p}$ and ${\mathscr{Y}}_{n,p}$ are exponentially equivalent as well. For this we recall the definition of $\zeta_2^{(n)}$ and write, for $\varepsilon>0$,
\begin{align*}
\Pro\big[|{\mathscr{Y}}_{n,p}-\widetilde{\mathscr{Y}}_{n,p}|>\varepsilon\big] = \Pro\bigg[\Big|\sum_{i=k_n+1}^n(g_i^2-1)\Big|>{2\varepsilon n t_n\over\sqrt{k_n}}\bigg].
\end{align*}
Denoting by $g$ a standard Gaussian random variable and applying a union bound we see that
\begin{align*}
\Pro\big[|{\mathscr{Y}}_{n,p}-\widetilde{\mathscr{Y}}_{n,p}|>\varepsilon\big] &\leq \Pro\Bigg[\sum_{i=k_n+1}^ng_i^2>{2\varepsilon n t_n\over \sqrt{k_n}}-(n-k_n)\Bigg]\\
&\leq (n-k_n)\Pro\Bigg[g>\sqrt{{2\varepsilon n t_n\over (n-k_n)\sqrt{k_n}}-1}\,\Bigg].
\end{align*}
But since $\Pro[g>x]\leq (x\sqrt{2\pi})^{-1}e^{-x^2/2}$ for any $x>0$, we conclude that
\begin{align*}
\lim_{n\to\infty}{1\over t_n^2}\Pro\big[|{\mathscr{Y}}_{n,p}-\widetilde{\mathscr{Y}}_{n,p}|>\varepsilon\big] &\leq -\lim_{n\to\infty}{1\over t_n^2}{2\varepsilon n t_n\over (n-k_n)\sqrt{k_n}} = - \lim_{n\to\infty}{2\varepsilon\over n-k_n}{\sqrt{k_n}\over t_n}{n\over k_n}.
\end{align*}
Since $n-k_n=o\big(\frac{\sqrt{k_n}}{t_n}\big)$, the last expression tends to $-\infty$, as $n\to\infty$. This proves the desired exponential equivalence.
\end{proof}

\subsection*{Acknowledgement}
Parts of the research presented in this work have been carried out during the workshop \textit{New Perspectives and Computational Challenges in High Dimensions} that took place at the Mathematical Research Institute Oberwolfach (MFO) in February 2020. All support for providing an excellent atmosphere and working conditions is gratefully acknowledged. JP is supported by the Austrian Science Fund (FWF) Project P32405 ``Asymptotic Geometric Analysis and Applications''.

\bibliographystyle{plain}
\bibliography{mdp}

\end{document}